\documentclass[a4paper,10pt,oneside, openany]{article}
\usepackage[a4paper, total={6in, 8in}]{geometry}
\usepackage[T1]{fontenc}
\usepackage[utf8]{inputenc}
\usepackage[italian, english]{babel}
\usepackage{microtype}
\usepackage{quoting}
\usepackage{amsthm, amsmath, amssymb, amscd, mathrsfs, mathtools} 
\usepackage{booktabs, caption, graphicx} 
\usepackage{pgfplots}
\usepackage{empheq}
\usepackage{braket}
\usepackage{emptypage}
\usepackage{amsfonts}
\usepackage{yfonts}
\usepackage{extarrows}
\usepackage{multicol}
\usepackage{changepage}
\usepackage{graphicx}
\usepackage{textcomp}
\usepackage[all]{xy}
\usepackage{qtree}
\usepackage{lipsum}
\usepackage{gensymb}
\usepackage{etoolbox}
\usepackage{authblk}
\usepackage{eso-pic}
\usepackage{minitoc}
\usepackage{verbatim}
\usepackage{pictexwd,dcpic}
\usepackage{tikz-cd}
\usetikzlibrary{arrows}
\usepackage{adjustbox}
\usepackage{eulervm}
\usepackage{blindtext}
\usepackage[sorting=nty,backend=biber, maxnames=10]{biblatex}
\AtEveryBibitem{%
  \clearfield{issn} 
  \clearfield{doi} 
  \clearfield{isbn}
  \clearfield{note}
  \clearfield{primaryClass}

  \ifentrytype{online}{}{
    \clearfield{url}
  }
}
\quotingsetup{font=small}

\theoremstyle{definition}
\newtheorem{Def}{Definition}[section]

\theoremstyle{plain}
\newtheorem{prop}[Def]{Proposition}

\theoremstyle{plain}
\newtheorem{cor}[Def]{Corollary}

\theoremstyle{plain}
\newtheorem{thm}[Def]{Theorem}

\theoremstyle{plain}
\newtheorem{lemma}[Def]{Lemma}

\theoremstyle{remark}
\newtheorem{oss}[Def]{Remark}

\theoremstyle{remark}
\newtheorem{esempio}[Def]{Example}

\theoremstyle{remark}

\theoremstyle{definition}
\newtheorem{Defi}{Definition}

\theoremstyle{plain}
\newtheorem{propi}[Defi]{Proposition}

\theoremstyle{plain}

\theoremstyle{plain}
\newtheorem{thmi}[Defi]{Theorem}

\theoremstyle{plain}

\theoremstyle{remark}

\theoremstyle{remark}

\theoremstyle{plain}

\theoremstyle{plain}

\DeclareMathOperator{\Hh}{H}

\newcommand{\id}{\text{id}}

\newcommand{\pic}[0]{\text{Pic}}
\newcommand{\pico}[1]{\text{Pic}(\overline{#1})}
\newcommand{\U}[0]{\text{U}}
\newcommand{\Ue}[0]{\emph{U}}
\newcommand{\eve}[0]{\emph{ev}}
\newcommand{\ev}[0]{\text{ev}}

\newcommand{\Addresses}{{
  \bigskip
  \footnotesize

  \textsc{Dipartimento di Matematica, Universit\`a di Pisa, Italy}\par\nopagebreak
  \textit{E-mail address}: \texttt{mattia.pirani@phd.unipi.it}
}}

\title{Flasque quasi-resolutions of algebraic varieties}
\date{}
\author{Mattia Pirani}

\begin{document}

\maketitle

\vspace{-0.7cm}

\begin{abstract}
Flasque resolutions play an important role in understanding birational properties of algebraic tori. For instance, Colliot-Th\'{e}l\`{e}ne and Sansuc have used them to compute $R$-equivalence classes of algebraic tori. We extend this notion to a larger class of algebraic varieties, including homogeneous spaces. This leads to a lower bound on the number of $R$-equivalence classes of homogeneous spaces, which is a slightly stronger version of a theorem of Colliot-Thélène and Kunyavski\u{\i}.  
\end{abstract}

\vspace{0.5cm}

\section*{Introduction}

The notions of flasque tori and flasque resolutions have been widely used for the investigation of birational questions about algebraic tori (see \cite{MR1634406}). One of the most relevant use of these techniques is due to Colliot-Th\'{e}l\`{e}ne and Sansuc. In \cite{requiv}, they made an intensive study of $R$-equivalence on algebraic tori. In particular, they found an explicit way to calculate $R$-equivalence of an algebraic torus $T$. Let 
\begin{equation*}
        \begin{tikzpicture}[baseline= (a).base]
        \node[scale=1] (a) at (0,0){
        \begin{tikzcd}
       1\arrow[r] & S\arrow[r] & E\arrow[r] & T\arrow[r] & 1
       \end{tikzcd}};
        \end{tikzpicture}
\end{equation*}
be a flasque resolution of $T$, that means $S$ is a flasque torus and $E$ is quasi-trivial. They proved that the connecting morphism produces an isomorphism:
\begin{equation*}
        \begin{tikzpicture}[baseline= (a).base]
        \node[scale=1] (a) at (0,0){
        \begin{tikzcd}
       T(k)/R\arrow[r, "\sim"] & \Hh^1(k,S),
       \end{tikzcd}};
        \end{tikzpicture}
\end{equation*}
where $T(k)/R$ denotes the set of classes of $R$-equivalence. In this way, they managed to prove that the number of classes of $R$-equivalence of an algebraic torus over some particular fields, like number fields, are finite by using cohomological finiteness theorems for flasque tori.

Further developments in this direction have been made by Colliot-Th\'{e}l\`{e}ne, in \cite{MR2404747}. He extended the notion of flasque resolutions to a connected linear algebraic group $G$:
\begin{equation*}
        \begin{tikzpicture}[baseline= (a).base]
        \node[scale=1] (a) at (0,0){
        \begin{tikzcd}
       1\arrow[r] & S\arrow[r] & \Tilde{G}\arrow[r] & G\arrow[r] & 1.
       \end{tikzcd}};
        \end{tikzpicture}
\end{equation*}
Here, the role of the quasi-trivial torus $E$ is taken by the $\mathbb G_m$-quasi-trivial group $\Tilde{G}$. Then, he proved that the connecting map is bijective when the field is good, that means $k$ has cohomological dimension $2$, and its finite extensions satisfy the index-period property and Serre's conjecture II.

The implications that these techniques had in the investigation of $R$-equivalence could suggest that an extension of flasque resolutions to a larger class of algebraic varieties could be a useful direction to explore. The aim of this paper is to introduce the concept of flasque quasi-resolutions, to investigate which properties from the classical case could be recovered and, at the end, we will use the language of flasque quasi-resolution to prove a slightly stronger version of a theorem of Colliot-Thélène and Kunyavski\u{\i}. 

The paper is structured as follow. The appendix is mainly dedicated to an exact sequence that Sansuc constructed in \cite{MR631309}. We are going to provide a proof that makes explicit the connecting morphism. Furthermore, we recall some well-known results about algebraic groups and torsors. In the first section, we give the definition of flasque quasi-resolution of an algebraic variety and we explore some general properties. The main results about flasque quasi-resolutions are proven in the second section. First of all, we investigate a special case, that is when $X$ is proper or a homogeneous space under a semi-simple group. In the second part, we partially extend the results to a larger class of varieties. We prove the following results about existence and uniqueness of flasque quasi-resolutions:

\begin{propi}
Let $X$ be a algebraic $k$-variety. Suppose that $X(k)$ is non-empty and fix a point $x\in X(k)$. There exists a flasque quasi-resolution of $(X,x)$ if and only if $\emph{Pic}(\overline{X})$ is finitely generated.
\end{propi}

\begin{propi}
Let $X$ be a algebraic $k$-variety. Suppose that $X(k)$ is non-empty and fix a point $x\in X(k)$. Let $Y_1\overset{S_1}{\longrightarrow} X$ and $Y_2\overset{S_2}{\longrightarrow} X$ be flasque quasi-resolutions of $(X,x)$. We can establish the following properties.
\begin{itemize}
    \item There exists an isomorphism of $k$-schemes $Y_1\times_k S_2\simeq Y_2\times_k S_1$. 
    \item There are two quasi-trivial tori $P_1$ and $P_2$, such that $P_1\times_k S_2\simeq P_2\times_k S_1$ as group schemes. 
    \item There exist two morphisms $S_1\rightarrow P_1$ and $S_2\rightarrow P_2$, that induce a bijection $\emph{Ker}(\mathrm H^n(k, S_1)\rightarrow \mathrm H^n(k,P_1))\simeq \emph{Ker}(\mathrm H^n(k, S_2)\rightarrow \mathrm H^n(k,P_2))$, for all non-negative integer $n$. 
\end{itemize}
As a consequence, $\mathrm H^1(k, S_1)\simeq \mathrm H^1(k, S_2)$.
\end{propi}

Given a flasque quasi-resolution $Y\overset{S}{\longrightarrow} X$, we can define a map by evaluating on $k$-points:
\begin{equation*}
        \begin{tikzpicture}[baseline= (a).base]
        \node[scale=1] (a) at (0,0){
        \begin{tikzcd}
        X(k) \arrow[r, "\text{ev}"] & \Hh^1(k,S). 
       \end{tikzcd}};
        \end{tikzpicture}
\end{equation*}
We will prove that for homogeneous spaces under linear algebraic $k$-group the evaluation map does not depend by the choice of the flasque quasi-resolution or the marked point. At the end of the section we extend the well-definition of the evaluation map to a larger class of algebraic varieties, like homogeneous spaces. The last section is dedicated to study the $R$-equivalence of homogeneous spaces. We provide a second construction of flasque quasi-resolutions. The advantage of this more explicit construction is that in this way it is easier to investigate the evaluation map. We conclude by proving a slightly generalization of a theorem of Colliot-Th\'{e}l\`{e}ne and Kunyavski\u{\i}:

\begin{thmi}
Let $X$ be a homogeneous space under a connected linear algebraic $k$-group with connected stabilizer. Let $x$ be a $k$-point of $X$. The evaluation map induced by a flasque quasi-resolution of $(X,x)$ is surjective if the field $k$ is good.
\end{thmi}

We provide two proofs of this result, both of them are based on the same idea of the original proof of Colliot-Th\'{e}l\`{e}ne and Kunyavski\u{\i}. This result naturally raises questions about the need to assume that the field $k$ is good. The author uses the techniques developed in this article to investigate on such a question, in \cite{pirani2023nonsurjectivity}. In particular, the author proves that, without assuming that the field $k$ is good, counterexamples can be constructed. For example, over fields of cohomological dimension $2$ and $2$-local fields such as $\mathbb Q_7((t))$.

\subsection*{Notations}

The letter $k$ will denote a field of characteristic zero. We fix an algebraic closure $\overline k$ and we denote by $\Gamma_k$ the absolute Galois group of $k$. An algebraic $k$-variety $X$ is a geometrically integral, separated scheme of finite type over $k$. We denote by $\overline X$ the base change $X\times_k \overline k$, and by $\U (X)$ the $\Gamma_k$-module $k[\overline X]^*/\overline k ^*$ (be careful that in Sansuc's paper $\U(X)$ is the abelian group $k[X]^*/ k ^*$). The letter $G$ will denote a linear algebraic $k$-group, that is a group scheme over $k$ whose underlying scheme is affine, separated, and of finite type over $k$. We denote by $\widehat G$ the $\Gamma_k$-module of characters $\text{Hom}_{\overline{k}-\text{grp}}(\overline{G}, \mathbb G_{m,\overline{k}})$. The letter $\Gamma$ will denote a (pro)finite group. Given a $\Gamma$-lattice $N$, we will denote its dual by $N^0$.

\subsection*{Acknowledgements}

I warmly thank my advisors Philippe Gille and Tamás Szamuely for their support, and for several and helpful discussions. The author acknowledges the MIUR Excellence Department Project awarded to the Department of Mathematics, University of Pisa, CUP I57G22000700001. The author is member of the Italian GNSAGA-INDAM. 

\section{Flasque tori and flasque quasi-resolutions}

The aim of this section is to give the definition of flasque (and coflasque) quasi-resolutions. We are going to investigate some general properties that follow easily from the definitions. But first, let us refer to classical literature, following \cite{requiv}. The letter $\Gamma$ denotes a profinite group.

\begin{Def}
A $\Gamma$-lattice $N$ is:
\begin{itemize}
    \item flasque, if $\Hh^1(\Gamma', N^0)=0$ for all open subgroups $\Gamma'$ of $\Gamma$;
    \item coflasque, if $N^0$ is flasque;
    \item permutation, if $N$ has a $\mathbb Z$-basis permuted by $\Gamma$;
    \item stably permutation, if there exist permutation lattices $P_1$ and $P_2$ such that $N\oplus P_1 \simeq P_2$.
\end{itemize}
\end{Def}

\begin{oss}
By Shapiro's lemma, all (stably) permutation lattices are both flasque and coflasque.
\end{oss}

\begin{oss}
Let $\Gamma'\rightarrow \Gamma$ be a morphism of profinite groups. If a $\Gamma$-lattice $N$ is flasque, coflasque, permutation or stably permutation as $\Gamma$-module, then it is as $\Gamma'$-module (see \cite[Remark 2]{requiv}). 
\end{oss}

\begin{oss}
Flasque lattices arise in a natural way from the geometry of certain algebraic $k$-varieties. Let $X$ be a homogeneous space under a connected linear algebraic $k$-group. Since the characteristic of $k$ is zero, we can fix a smooth compactification $X_c$ of $X$, because of Hironaka (see \cite{MR0199184}). In \cite[Theorem 5.1]{MR2237268}, Colliot-Th\'{e}l\`{e}ne and Kunyavski\u{\i} established that $\pico{X_c}$ is a flasque $\Gamma_k$-lattice, extending previous results on algebraic tori, due to Voskresenski\u{\i} (see \cite[Section 4.6]{MR1634406}), and on connected linear algebraic $k$-groups, due to Borovoi and Kunyavski\u{\i} (see \cite[Theorem 3.2]{MR2054399}). 
\end{oss}

We can define an equivalence relation on the set of lattices:

\begin{Def}
Let $N_1$ and $N_2$ be $\Gamma$-lattices. We say they are similar if there are two permutation lattices $P_1$ and $P_2$, such that $N_1\oplus P_1 \simeq N_2\oplus P_2$. 
\end{Def}

We recall the existence of flasque and coflasque resolutions: 

\begin{prop}
\label{fame}
Let $N$ be a $\Gamma$-module. Suppose that $N$ is finitely generated as abelian group. 
\begin{itemize}
    \item There exists a flasque resolution
        \begin{equation*}
        \begin{tikzpicture}[baseline= (a).base]
        \node[scale=1] (a) at (0,0){
        \begin{tikzcd}
       1\arrow[r] & P \arrow[r] & F \arrow[r] & N \arrow[r] & 1,
       \end{tikzcd}};
        \end{tikzpicture}
\end{equation*}
    where $P$ is a permutation lattice and $F$ is a flasque lattice.
    \item There exists a coflasque resolution
        \begin{equation*}
        \begin{tikzpicture}[baseline= (a).base]
        \node[scale=1] (a) at (0,0){
        \begin{tikzcd}
       1\arrow[r] & C \arrow[r] & P \arrow[r] & N \arrow[r] & 1,
       \end{tikzcd}};
        \end{tikzpicture}
\end{equation*}
    where $P$ is a permutation lattice and $C$ is a coflasque lattice.
\end{itemize}
\end{prop}

\begin{proof}
See \cite[Lemma 3]{requiv} or \cite[Lemma 0.6]{MR878473}.
\end{proof}

\begin{oss}
If $N$ is a $\Gamma$-lattice, then there exist dual versions of the exact sequences of previous proposition. 
\end{oss}

The following proposition shows some sort of uniqueness of flasque and coflasque resolutions:

\begin{prop}
Let $1\rightarrow P_1\rightarrow F \rightarrow N\rightarrow 1$ and $1\rightarrow C\rightarrow P_2 \rightarrow N\rightarrow 1$ be flasque and coflasque resolutions of $N$, respectively. The similarity classes of $F$ and $C$ depends only on $N$. 
\end{prop}

\begin{proof}
See \cite[Lemma 5]{requiv}.
\end{proof}

The analogous definition for algebraic tori is the following:

\begin{Def}
Let $T$ be an algebraic torus. We say that $T$ is flasque, coflasque or quasi-trivial, if the $\Gamma_k$-lattice $\widehat T$ is flasque, coflasque or quasi-trivial, respectively.
\end{Def}

\begin{oss}
In analogy to flasque and coflasque resolutions of $\Gamma_k$-modules, we can define them for groups of multiplicative type over $k$.
\end{oss}

The following definition has been introduced by Colliot-Th\'{e}l\`{e}ne. He used it to extend the notion of flasque and coflasque resolutions to connected linear algebraic $k$-groups, in \cite[Proposition-Definition 3.1]{MR2404747} and \cite[Proposition 4.1]{MR2404747}. We are going to do the same for flasque quasi-resolutions.

\begin{Def}
Let $Y$ be an algebraic $k$-variety such that $\pic(\overline{Y})$ is trivial. We say that $Y$ is:
\begin{itemize}
    \item $\mathbb G_m$-quasi-trivial, if $\U(Y)$ is a permutation lattice;
    \item $\mathbb G_m$-coflasque, if $\U(Y)$ is a coflasque lattice.
\end{itemize}
\end{Def}

\begin{oss}
For an algebraic $k$-torus $T$, it is equivalent to be $\mathbb G_m$-quasi-trivial and quasi-trivial, and the same is true for $\mathbb G_m$-coflasque and coflasque. Indeed, the group $\text{Pic}(\overline T)$ is always trivial and the $\Gamma_k$-lattice $\U (T)$ is isomorphic to $\widehat T$, by Proposition \ref{rosigene}.
\end{oss}

In analogy to the tori case, there is the following rigidity property:

\begin{prop}
\label{concentrato}
Let $X$ be an algebraic $k$-variety, and let $T$ be an algebraic torus.
\begin{itemize}
    \item If $X$ is $\mathbb G_m$-quasi-trivial and $T$ is a flasque torus, then $T$-torsors over $X$ are constant.
    \item If $X$ is $\mathbb G_m$-coflasque and $T$ is a quasi-trivial torus, then $T$-torsors over $X$ are trivial. 
\end{itemize}
\end{prop}

\begin{proof}
See \cite[Proposition 1.3]{MR2404747} and \cite[Proposition 1.7]{MR2404747}.
\end{proof}

\subsection{Definitions and general properties}

The definition of a flasque quasi-resolutions mirrors Colliot-Th\'{e}l\`{e}ne's idea for flasque resolutions of groups. We will indeed see that for groups, the two notions coincide. We will provide the definition for every variety, but we will see that it only makes sense under certain assumptions.

\begin{Def}
\label{defmia}
Let $X$ be an algebraic $k$-variety with $X(k)\neq \emptyset$. We fix a point $x\in X(k)$. A torsor $Y\longrightarrow X$ under a flasque torus $S$ is a flasque quasi-resolution of $(X,x)$ if the following conditions are satisfied:
\begin{itemize}
    \item the $k$-scheme $Y$ is a $\mathbb G_m$-quasi-trivial algebraic $k$-variety;
    \item the fibre over $x$ has a $k$-rational point (i.e. the torsor over $x$ is the trivial one).
\end{itemize}
\end{Def}

\begin{Def}
\label{defmia2}
Let $X$ be an algebraic $k$-variety with $X(k)\neq \emptyset$. A torsor $Y\longrightarrow X$ under a quasi-trivial torus $P$ is a coflasque quasi-resolution of $X$ if $Y$ is a $\mathbb G_m$-coflasque algebraic $k$-variety.
\end{Def}

\begin{oss}
In the definition of coflasque quasi-resolutions there is no mark point. Indeed, since $P$ is quasi-trivial, the torsors over the $k$-points of $X$ are the trivial ones.
\end{oss}

A necessary condition for the existence of flasque (and coflasque) quasi-resolutions is provided by the following proposition:

\begin{prop}
Let $X$ be a smooth algebraic $k$-variety with $X(k)\neq \emptyset$. If there exists a flasque (or coflasque) quasi-resolution of $X$, then $\emph{Pic}(\overline{ X})$ is finitely generated.
\end{prop}

\begin{proof}
    Let $Y\overset{T}{\longrightarrow} X$ be a flasque (or coflasque) quasi resolution of $X$. Applying the exact sequence of Sansuc to this torsor, we obtain the following exact sequence:
    \begin{equation*}
        \begin{tikzpicture}[baseline= (a).base]
        \node[scale=1] (a) at (0,0){
        \begin{tikzcd}
       \widehat T \arrow[r] & \pico X \arrow[r] & \pico Y=0 .
       \end{tikzcd}};
        \end{tikzpicture}
    \end{equation*}
    The group $\widehat T$ is finitely generated, and therefore, $\pico X$ is also finitely generated.
\end{proof}

\begin{oss}
Let $X$ be a smooth algebraic $k$-variety. As a corollary of Hironaka's construction, a smooth compactification $X_c$ of $X$ exists. The group $\pico X$ is finitely generated if and only if $\pico {X_c}$ is finitely generated. This equivalence can be established through the following exact sequence:
\begin{equation*}
        \begin{tikzpicture}[baseline= (a).base]
        \node[scale=1] (a) at (0,0){
        \begin{tikzcd}
        0\arrow[r]&\text{U}(X)\arrow[r]&\text{Div}_{\overline{Z}}{\overline{X_c}}\arrow[r]&\pico {X_c} \arrow[r]& \pico X\arrow[r]&0,
       \end{tikzcd}};
        \end{tikzpicture}
\end{equation*}
where the Galois module $\text{Div}_{\overline{Z}}{\overline{X_c}}$ represents the group of divisors with support in the closed subset $\overline{Z}=\overline{X_c}\backslash \overline{X}$. This $\text{Div}_{\overline{F}}{\overline{Z}}$ is a permutation lattice.
\end{oss}

\begin{oss}
Let $X$ be a smooth algebraic $k$-variety. If $X$ is $\overline{k}$-unirational (for example, in the case of homogeneous spaces), then we can establish that $\pico {X_c}$ is finitely generated and free. It is well known fact that the N\'{e}ron-Severi group is finitely generated. To prove our claim, we aim to show that $\text{Pic}^0(\overline{X_c})$ is zero. This group is an abelian variety, and it can either be zero or have non-trivial $n$-torsion for all $n$. By \cite[Proposition 2]{MR109155}, the fundamental group $\pi_1(\overline{X_c})$ is finite. If $n$ is coprime with the order of $\pi_1(\overline{X_c})$, then the $n$-torsion of $\pico {X_c}$ is equal to $\text{H}^1(\overline{X_c},\mu_n)=\text{H}^1(\overline{X_c},\mathbb Z/n\mathbb Z)=0$. This demonstrates that $\text{Pic}^0(\overline{X_c})$ is indeed zero. Furthermore, by \cite[Proposition 3]{MR109155}, the N\'{e}ron-Severi group is a free abelian group. Consequently, we can conclude that $\text{Pic}(\overline{X_c})$ is also free.
\end{oss}

Let $G$ be a connected linear algebraic $k$-group. For $G$ we can compare flasque quasi-resolutions and flasque resolutions of Colliot-Th\'{e}l\`{e}ne. We aim to show that these two definitions coincide in this case. In defining flasque quasi-resolutions, it is necessary to select a base point in $G(k)$. A natural choice for this base point is the identity element, denoted by $1_G \in G(k)$.

\begin{prop}
Let $G$ be a connected linear algebraic $k$-group. 
\begin{itemize}
    \item If the central exact sequence 
    \begin{equation*}
        \begin{tikzpicture}[baseline= (a).base]
        \node[scale=1] (a) at (0,0){
        \begin{tikzcd}
       1 \arrow[r]& S \arrow[r] & H \arrow[r] & G \arrow[r] & 1,
       \end{tikzcd}};
        \end{tikzpicture}
    \end{equation*}
    is a flasque resolution of $G$ in the sense of Colliot-Th\'{e}l\`{e}ne, then the $S$-torsor $H\longrightarrow G$ is a flasque quasi-resolution of $(G,1_G)$. 
    \item If $H\overset{S}{\longrightarrow} G$ is a flasque quasi-resolution of $(G,1_G)$, then $H$ has a structure of connected linear algebraic $k$-group. Moreover, the $S$-torsor $H\longrightarrow G$ fits into a central exact sequence
    \begin{equation*}
        \begin{tikzpicture}[baseline= (a).base]
        \node[scale=1] (a) at (0,0){
        \begin{tikzcd}
       1 \arrow[r]& S \arrow[r] & H \arrow[r] & G \arrow[r] & 1,
       \end{tikzcd}};
        \end{tikzpicture}
    \end{equation*}
    that is a flasque resolution of $G$.
\end{itemize}
\end{prop}

\begin{proof}
The first point is obvious. The second point is a theorem due to Colliot-Th\'{e}l\`{e}ne (see \cite[Theorem 5.6]{MR2404747}).
\end{proof}

\begin{oss}
There is a similar result for coflasque quasi-resolutions as well.
\end{oss}



Flasque (and coflasque) quasi-resolutions of a $\mathbb G_m$-quasi-trivial algebraic $k$-variety are very rigid. The following propositions, drawing an analogy to the case of tori, provide some kind of classification for them:

\begin{prop}
\label{flasquasitrivia}
Let $X$ be a $\mathbb G_m$-quasi-trivial algebraic $k$-variety with $X(k)\neq \emptyset$. Fix a point $x\in X(k)$. 
\begin{itemize}
    \item Let $P$ be a quasi-trivial torus. The $P$-torsor $X\times_k P\longrightarrow X$ is a flasque quasi-resolution of $(X,x)$.
    \item All flasque quasi-resolutions of $(X,x)$ are of the kind $X\times_k S\longrightarrow X$, where $\widehat S$ is a stably permutation lattice. 
\end{itemize}
\end{prop}

\begin{proof}
The first item is a computation. Specifically, consider $X\times_k P \longrightarrow X$ as described in the statement. By Lemma \ref{rosi}, we establish that $\U(X\times_k P)\simeq \U(X)\oplus \widehat P$, and applying Lemma \ref{san}, we find that $\pico {X\times_kP}\simeq \pico X \oplus \pico P=0$. This implies that $X\times_k P$ is a $\mathbb G_m$-quasi-trivial algebraic $k$-variety. The remaining properties are clear, therefore the $P$-torsor is a flasque quasi-resolution of $(X,x)$. Now, let $Y\overset{S}{\longrightarrow}X$ be a flasque quasi-resolution of $(X,x)$. According to Proposition \ref{concentrato}, the torsor is constant, thus it is isomorphic to $X\times_k S\longrightarrow X$ since $Y(k)$ is non-empty. By Lemma \ref{rosi}, we deduce that $\U(X\times_k S)\simeq \U(X)\oplus \widehat S$ and thus $\widehat S$ is a stably permutation lattice.
\end{proof}

The same result holds for coflasque quasi-resolutions, and the proof is similar.

\begin{prop}
\label{flasquasitrivia2}
Let $X$ be a $\mathbb G_m$-coflasque algebraic $k$-variety with $X(k)$ non-empty. Let $P$ be a quasi-trivial torus. The $P$-torsor $X\times_k P\longrightarrow X$ is a coflasque quasi-resolution of $X$. Furthermore, all coflasque quasi-resolutions of $X$ fall into this category.
\end{prop}

The following proposition establishes a form of `versal' property for flasque quasi-resolutions, reminiscent of a well-known analogous statement for tori. 

\begin{prop}
\label{versal}
Let $X$ be an algebraic $k$-variety with $X(k)\neq \emptyset$. Fix a point $x \in X(k)$. Let $Y\overset{S}{\longrightarrow} X$ be a flasque quasi-resolution of $(X,x)$. Any $k$-morphism $Z\rightarrow X$ that verifies the following conditions:
\begin{itemize}
    \item the algebraic $k$-variety $Z$ is $\mathbb G_m$-quasi-trivial,
    \item the fibre over $x$ contains a point $z\in Z(k)$,
\end{itemize}
factors as follows:
\begin{equation*}
        \begin{tikzpicture}[baseline= (a).base]
        \node[scale=1] (a) at (0,0){
        \begin{tikzcd}
        & Y \arrow[d, "S"]\\
        Z \arrow[r]\arrow[ur, "\psi"]& X.
       \end{tikzcd}};
        \end{tikzpicture}
\end{equation*}
Furthermore, if we choose a point $y \in Y(k)$ in the fibre over $x$, then we can take $\psi$ such that $\psi (z)=y$.
\end{prop}

\begin{proof}
We consider the fibre product:
\begin{equation*}
        \begin{tikzpicture}[baseline= (a).base]
        \node[scale=1] (a) at (0,0){
        \begin{tikzcd}
        Y_Z \arrow[r] \arrow[d, "S"]& Y \arrow[d, "S"]\\
        Z \arrow[r] & X.
       \end{tikzcd}};
        \end{tikzpicture}
\end{equation*}
The fibre over the point $z$ is trivial. According to Proposition \ref{concentrato}, the $S$-torsor $Y_Z\longrightarrow Z$ is constant and, it is the trivial torsor since $Y_Z$ contains a $k$-point. Thus, there exists a map $Z\rightarrow Y$ that makes the diagram commute. By modifying the section $Z\rightarrow Y_Z$, we can ensure that $\psi (z)=y$. 
\end{proof}

Using a similar proof technique, we can establish the following statement. 

\begin{prop}
\label{versal2}
Let $X$ be an algebraic $k$-variety with $X(k)\neq \emptyset$. Let $Y\overset{P}{\longrightarrow} X$ be a coflasque quasi-resolution of $X$. Let $Z$ be a $\mathbb G_m$-coflasque algebraic $k$-variety. Any $k$-morphism $Z\rightarrow X$ factors as follows: 
\begin{equation*}
        \begin{tikzpicture}[baseline= (a).base]
        \node[scale=1] (a) at (0,0){
        \begin{tikzcd}
        & Y \arrow[d, "P"]\\
        Z \arrow[r]\arrow[ur, "\psi"]& X.
       \end{tikzcd}};
        \end{tikzpicture}
\end{equation*}
Furthermore, if we choose a point $y \in Y(k)$ in the fibre of $x$, then we can take $\psi$ such that $\psi (z)=y$.
\end{prop}

\begin{oss}
Under the assumption that the algebraic varieties have structure of linear algebraic groups, and the morphisms preserve the group structure, with $z$ and $y$ being the group identities, we can demonstrate that $\psi$ also preserves the group structure. The proof follows the same idea to as one found in \cite[Proposition 3.2]{MR2404747}.
\end{oss}

It is natural to explore conditions that determine when the pull-back of a flasque quasi-resolution remains a flasque quasi-resolution. The following proposition gives a partial result in that direction. 

\begin{prop}
\label{colliotprati}
Let $X$ be a smooth algebraic $k$-variety with $X(k)\neq \emptyset$. Fix a point $x \in X(k)$. Let $Y\overset{S}{\longrightarrow} X$ be a flasque quasi-resolution of $(X,x)$. Let $f:X'\rightarrow X$ be a morphism of $k$-schemes such that the fibre over $x$ has a $k$-rational point $x'$. Assume that $f$ satisfies one of the following properties:
\begin{itemize}
    \item $f$ is an open embedding.
    \item $f$ is a $G$-torsor, where $G$ is a $\mathbb G_m$-quasi-trivial linear algebraic $k$-group.
\end{itemize}
Then the base change $Y'\overset{S}{\longrightarrow} X'$ is a flasque quasi-resolution of $(X',x')$.
\end{prop}

\begin{proof}
We prove that $Y'$ is $\mathbb G_m$-quasi-trivial. 
\begin{itemize}
    \item If $f$ is an open embedding, then $Y'$ is an open subset of $Y$ and we can apply \cite[Proposition 1.2]{MR2404747}.
    \item If $f$ is a $G$-torsor, we apply the Sansuc exact sequence to the $G$-torsor $Y'\longrightarrow Y$:
    \begin{equation*}
        \begin{tikzpicture}[baseline= (a).base]
        \node[scale=1] (a) at (0,0){
        \begin{tikzcd}
        0\arrow[r]&\U(Y)\arrow[r]&\U(Y')\arrow[r]&\widehat{G}\arrow[r]& \text{Pic}(\overline{Y})\arrow[r]& \text{Pic}(\overline{Y'})\arrow[r]& \text{Pic}(\overline{G}).
       \end{tikzcd}};
        \end{tikzpicture}
    \end{equation*}
    By hypothesis, $\U(Y)$ and $\widehat G$ are permutation modules, and $\pico Y$ and $\pico G$ are trivial. Consequently, we can deduce that that $\pico {Y'}$ is trivial, and $\U(Y')\simeq \U(Y)\oplus \widehat G$ is a permutation module.
\end{itemize}
To conclude, we observe that, by assumption, there exists a rational point in the fibre above $x'$ of $Y'\longrightarrow X'$. Therefore, such an $S$-torsor is a flasque quasi-resolution.
\end{proof}

\begin{oss}
There is a similar result for coflasque quasi-resolutions as well.
\end{oss}



Under appropriate assumptions flasque quasi-resolutions behave well with respect to the product.

\begin{prop}
\label{prodotto}
Let $X$ and $X'$ be smooth algebraic $k$-varieties, both of which have a $k$-rational point. Choose two points $x\in X(k)$ and $x'\in X'(k)$. Consider two flasque quasi-resolutions $Y\overset{S}{\longrightarrow} X$ and $Y'\overset{S'}{\longrightarrow} X'$ of $(X,x)$ and $(X',x')$, respectively. Assuming that $X$ is $\overline{k}$-rational, the map $Y\times_k Y'\longrightarrow X\times_k X'$ is a flasque quasi-resolution of $(X\times_k X', (x,x'))$.
\end{prop}

\begin{proof}
By Lemma \ref{rosi} and Lemma \ref{san}, the module $\U(Y\times_k Y')\simeq \U(Y)\oplus \U(Y')$ is a permutation module, and $\pico {Y\times_k Y'}\simeq\pico Y \oplus \pico {Y'}=0$. Due to the construction, the $S\times_k S'$-torsor $Y\times_k Y'\longrightarrow X\times_k X'$ is trivial at the point $x$. Therefore, we can establish the claim.
\end{proof}

\begin{oss}
There is a similar result for coflasque quasi-resolutions as well.
\end{oss}



We conclude this section by examining a (uni)rationality result. 

\begin{prop}
\label{sistemaresult}
Let $X$ be an algebraic $k$-variety with $X(k)\neq \emptyset$. Consider a morphism $\phi:U\rightarrow X$, where $U$ is an open subset of $\mathbb P^n_k$. Let $x$ be a point in $\phi\left(U(k)\right)$. Let $Y\overset{S}{\longrightarrow}X$ be a flasque quasi-resolution of $(X,x)$.
\begin{itemize}
    \item If $\phi$ has dense image, then $Y$ is $k$-unirational.
    \item If $X$ is smooth, and $\phi$ restricts to a $k$-isomorphism in a neighbourhood of $x$, then $Y$ is $k$-stably rational, and $\widehat S$ is stably permutation. 
\end{itemize}
\end{prop}

\begin{proof}
We consider the fibre product:
\begin{equation*}
        \begin{tikzpicture}[baseline= (a).base]
        \node[scale=1] (a) at (0,0){
        \begin{tikzcd}
       Y_U\arrow[r]\arrow[d,"S"]& Y \arrow[d, "S"]\\
       U\arrow[r, "\phi"] & X.
       \end{tikzcd}};
        \end{tikzpicture}
    \end{equation*}
 By \cite[Corollary 2.6]{MR878473}, the $S$-torsor $Y_U \longrightarrow U$ is constant. Since $x \in \phi(U(k))$, the set $Y_U(k)$ is non-empty, and we conclude that the torsor is trivial.
\begin{itemize}
    \item Assume that $\phi$ has dense image. Then $Y_U\simeq U\times_k S\longrightarrow Y$ has dense image, therefore $Y$ is $k$-unirational
    \item By hypothesis, without loss of generality, we can consider a smaller $U$ such that the map is an open embedding, and $x \in U$. By Proposition \ref{colliotprati}, the $S$-torsor $Y_U\longrightarrow U$ is a flasque quasi-resolution of $(U,x)$. In particular, $\U(Y_U)\simeq \U(U)\oplus \widehat S$ is a permutation lattice, therefore $\widehat S$ is a stably permutation module. Since $Y_U\simeq U\times_k S$ is an open subset of $Y$ we obtain that $Y$ is stably $k$-rational.
\end{itemize}
\end{proof}

\begin{oss}
\label{remarksenzanome}
The second item of the previous proposition has an analogue when $X$ is $k$-stably rational. The proof remains the same, and $\widehat S$ is a stably permutation lattice under this assumption, too.
\end{oss}

\begin{oss}
The previous proposition also yields a similar result concerning coflasque quasi-resolutions. However, in this case, the proof is simplified because the torsor is locally trivial for the Zariski topology.
\end{oss}

\section{Existence and uniqueness of flasque quasi-resolutions}

In this section, we state the main results on flasque quasi-resolutions. First of all, we see a special case, like the one of proper $k$-unirational varieties or homogeneous space under a semi-simple group. Then, we pass to the general case. 

In this section, we assume that $X$ is a smooth algebraic $k$-variety. We assume that $X(k)$ is non-empty and that $\pico X$ is finitely generated. 

\subsection{Special case}
\label{special}

Let $X$ be an algebraic $k$-variety as at the beginning of the section. We assume that $\U(X)$ is trivial. We can state the following result due to Colliot-Th\'{e}l\`{e}ne and Sansuc:

\begin{prop}
\label{sequesa}
Let $M$ be a group of multiplicative type defined over $k$. There exists a functorial exact sequence of abelian groups:
\begin{equation*}
        \begin{tikzpicture}[baseline= (a).base]
        \node[scale=1] (a) at (0,0){
        \begin{tikzcd}
        0\arrow[r]& \emph{H}^1(k,M)\arrow[r]& \emph{H}^1(X,M)\arrow[r, "t"]&\emph{Hom}_{\Gamma_k}(\widehat M, \emph{Pic}(\overline{X}))\arrow[r]& 0,
       \end{tikzcd}};
        \end{tikzpicture}
\end{equation*}
where the first map is induced by the structural morphism and $t$ is the type map. Furthermore, each element of $X(k)$ induces a section of the first map.
\end{prop}

\begin{proof}
See \cite[Section 2]{sketch} or \cite[Corollary 2.3.9]{skorobogatov_2001}).
\end{proof}

\begin{oss}
Let $X$ and $M$ be as in Proposition \ref{sequesa}. We fix a point $x\in X(k)$. The non-canonical isomorphism
\begin{equation*}
        \begin{tikzpicture}[baseline= (a).base]
        \node[scale=1] (a) at (0,0){
        \begin{tikzcd}
        \text{H}^1(X,M)\arrow[r, "\sim"]&\text{Hom}_{\Gamma_k}\left(\widehat M, \text{Pic}(\overline{X})\right)\oplus \text{H}^1(k,M)
       \end{tikzcd}};
        \end{tikzpicture}
\end{equation*}
implies that an $M$-torsor over $X$ is uniquely determined, up to isomorphism, by its type and the fibre over $x$.
\end{oss}

We can give the following definition from \cite[Section 2]{sketch}:

\begin{Def}
\label{tipo}
Let $M$ be the group of multiplicative type with $\widehat M=\text{Pic}(\overline X)$. We say that an $M$-torsor over $X$ is universal if its $type$ is the identity map.
\end{Def}

\begin{oss}
Suppose that $X(k)$ is non-empty. Once that we fix a point $x\in X(k)$, the universal torsor is unique, up to unique isomorphism.
\end{oss}

For the remaining part of this subsection, we fix a point $x\in X(k)$ and denote $\pico X$ as $\widehat M$, for some group of multiplicative type $M$. First of all, we prove a lemma:

\begin{lemma}
\label{nonso}
Let $S$ be a flasque torus, and let $Y\longrightarrow X$ be an $S$-torsor that is trivial over $x$. The following statements are equivalent:
\begin{itemize}
    \item The torsor is a flasque quasi-resolution of $(X,x)$.
    \item The type $\widehat S \longrightarrow \emph{Pic}(\overline{ X})$ is surjective and the kernel is a permutation lattice.
\end{itemize}
\end{lemma}

\begin{proof}
We can apply the exact sequence of Sansuc to the $S$-torsor $Y\longrightarrow X$. Since $\U(X)$ and $\pico S$ are trivial, we obtain 
\begin{equation*}
        \begin{tikzpicture}[baseline= (a).base]
        \node[scale=1] (a) at (0,0){
        \begin{tikzcd}
        0\arrow[r]&\U(Y)\arrow[r]&\widehat{S}\arrow[r]& \text{Pic}(\overline{X})\arrow[r]& \text{Pic}(\overline{Y})\arrow[r]& 0.
       \end{tikzcd}};
        \end{tikzpicture}
\end{equation*}
The lemma follows.
\end{proof}

\begin{oss}
A similar result holds for coflasque quasi-resolutions as well.
\end{oss}

By Proposition \ref{sequesa}, there exists a universal $M$-torsor $Z_0 \longrightarrow X$ which is trivial over $x$. Let 
\begin{equation*}
        \begin{tikzpicture}[baseline= (a).base]
        \node[scale=1] (a) at (0,0){
        \begin{tikzcd}
        1\arrow[r]& M\arrow[r]&S_0\arrow[r]&E_0\arrow[r]&1
       \end{tikzcd}};
        \end{tikzpicture}
\end{equation*}
be a flasque resolution. We denote by $Y_0 \longrightarrow X$ the pushforward of the $M$-torsor $Z_0\longrightarrow X$ by the map $M\rightarrow S_0$. We have just provided a method to construct flasque quasi-resolutions:

\begin{prop}
\label{vera}
The $S_0$-torsor $Y_0\longrightarrow X$ is a flasque quasi-resolution of $(X,x)$.
\end{prop}

\begin{proof}
By construction, the universal torsor is trivial over $x$. This implies that the pushforward by $M\rightarrow S_0$ is also trivial over $x$. It is enough to prove that the type $\widehat {S_0}\longrightarrow \text{Pic}(\overline{ X})$ is surjective and the kernel is a permutation lattice, by Lemma \ref{nonso}. By Proposition \ref{sequesa}, there is the following commutative diagram:
\begin{equation*}
        \begin{tikzpicture}[baseline= (a).base]
        \node[scale=1] (a) at (0,0){
        \begin{tikzcd}
        0\arrow[r] & \text{H}^1(k,S_0)\arrow[r] & \text{H}^1(X, S_0)\arrow[r] & \text{Hom}_{\Gamma_k}(\widehat {S_0}, \widehat M)\arrow[r] & 0\\
         0\arrow[r] & \text{H}^1(k,M)\arrow[r]\arrow[u] & \text{H}^1(X, M)\arrow[r]\arrow[u] & \text{Hom}_{\Gamma_k}(\widehat M, \widehat M)\arrow[r]\arrow[u] & 0,
       \end{tikzcd}};
        \end{tikzpicture}
\end{equation*}
where the first two vertical arrows are the pushforward by $M\rightarrow {S_0}$, and the last one is the pull-back by the map $\widehat {S_0}\rightarrow \widehat M$. By diagram chasing, we can see that the type of the $S_0$-torsor $Y_0\longrightarrow X$ is the composition:
\begin{equation*}
        \begin{tikzpicture}[baseline= (a).base]
        \node[scale=1] (a) at (0,0){
        \begin{tikzcd}
        \widehat {S_0} \arrow[r] & \widehat M \arrow[r, "\id"] & \widehat M ,
       \end{tikzcd}};
        \end{tikzpicture}
\end{equation*}
where the first map is the one of the flasque resolution $\widehat {S_0}\rightarrow \widehat M$ and the second map is the type of the universal torsor. By applying the exact sequence of Sansuc, we can deduce that the kernel $\U(Y_0)$ is a permutation lattice, and $\pico{Y_0}$ is zero.
\end{proof}

The previous construction demonstrates how to create flasque quasi-resolutions in this context. Now, we aim to establish that all flasque quasi-resolutions, up to isomorphism, are obtained in this way. 

\begin{prop}
\label{uni}
Let $Z\longrightarrow X$ be the universal $M$-torsor that is trivial over $x$, and let $Y\overset{S}{\longrightarrow}X$ be a flasque quasi-resolution of $(X,x)$. There exists a flasque resolution:
\begin{equation*}
        \begin{tikzpicture}[baseline= (a).base]
        \node[scale=1] (a) at (0,0){
        \begin{tikzcd}
        1\arrow[r] & M\arrow[r] & S\arrow[r] & E\arrow[r] & 1
        \end{tikzcd}};
        \end{tikzpicture}
\end{equation*}
such that the pushforward of the universal torsor by $M\rightarrow S$ is isomorphic to $Y\overset{S}{\longrightarrow}X$.
\end{prop}

\begin{proof}
By applying Lemma \ref{nonso}, we obtain a flasque resolution: 
\begin{equation*}
        \begin{tikzpicture}[baseline= (a).base]
        \node[scale=1] (a) at (0,0){
        \begin{tikzcd}
        0\arrow[r]&\widehat E\arrow[r]&\widehat{S}\arrow[r]& \widehat M\arrow[r]& 0,
       \end{tikzcd}};
        \end{tikzpicture}
\end{equation*}
where $\widehat S\rightarrow \widehat M$ is the type of the torsor. By Proposition \ref{sequesa}, there is the following commutative diagram with split exact sequences:
\begin{equation*}
        \begin{tikzpicture}[baseline= (a).base]
        \node[scale=1] (a) at (0,0){
        \begin{tikzcd}
        0\arrow[r] & \text{H}^1(k,S)\arrow[r] & \text{H}^1(X, S)\arrow[r] & \text{Hom}_{\Gamma_k}(\widehat {S}, \widehat M)\arrow[r] & 0\\
         0\arrow[r] & \text{H}^1(k,M)\arrow[r]\arrow[u] & \text{H}^1(X, M)\arrow[r]\arrow[u] & \text{Hom}_{\Gamma_k}(\widehat M, \widehat M)\arrow[r]\arrow[u] & 0.
       \end{tikzcd}};
        \end{tikzpicture}
\end{equation*}
We can conclude that the $S$-torsor $Y\longrightarrow X$ is isomorphic to the pushforward of the $M$-torsor $Z\longrightarrow X$ by $M\rightarrow S$. Indeed, by diagram chasing in the right square, we can deduce that these two $S$-torsors have the same type. Furthermore, since they are both trivial over $x$, they are isomorphic. 
\end{proof}

\begin{oss}
There is a similar result for coflasque quasi-resolutions as well.
\end{oss}



In the remaining part of this section, we will be only interested in flasque quasi-resolutions, as the analogous results for coflasque quasi-resolutions are straightforward and less interesting. 

We fix two points $x_1, x_2 \in X(k)$. Our next objective is to understand how flasque quasi-resolutions of $(X,x_1)$ and $(X, x_2)$ are related. 

Let $Y_1\overset{S}{\longrightarrow}X$ be a flasque quasi-resolution of $(X,x_1)$, and let $ Y_{1,x_2}\longrightarrow \text{Spec}\ k$ be the fibre over $x_2$. We can define a bijection by subtracting the class of the $S$-torsor $X \times_k Y_{1,x_2}\longrightarrow X$:
\begin{equation*}
        \begin{tikzpicture}[baseline= (a).base]
        \node[scale=1] (a) at (0,0){
        \begin{tikzcd}
        \text{H}^1(X,S) \arrow[r, "\psi"] & \text{H}^1(X,S).
       \end{tikzcd}};
        \end{tikzpicture}
\end{equation*}
We denote a representative of the image of the $S$-torsor $Y_1\longrightarrow X$ by $\psi$ as $Y_2\longrightarrow X$. 

\begin{prop}
\label{corresp} 
The $S$-torsor $Y_2\longrightarrow X$ is a flasque quasi-resolution of $(X,x_2)$. Furthermore, the previous construction establishes a bijection, up to isomorphism, between flasque quasi-resolutions of $(X,x_1)$ and flasque quasi-resolutions of $(X,x_2)$.
\end{prop}

\begin{proof}
There is a commutative diagram 
\begin{equation*}
        \begin{tikzpicture}[baseline= (a).base]
        \node[scale=1] (a) at (0,0){
        \begin{tikzcd}
        \text{H}^1(X,S)\arrow[r, "\psi"]\arrow[d] & \text{H}^1(X,S)\arrow[d] \\
        \text{H}^1(k,S) \arrow[r, "\psi'"] & \text{H}^1(k,S),
       \end{tikzcd}};
        \end{tikzpicture}
\end{equation*}
where the vertical arrows are the evaluations in $x_2$, and $\psi'$ is the subtraction by the $S$-torsor $Y_{1,x_2}\longrightarrow \text{Spec}\ k$. This diagram shows that the $S$-torsor $Y_2\longrightarrow X$ is trivial over $x_2$. According to Proposition \ref{sequesa}, there exists an exact sequence:
\begin{equation*}
        \begin{tikzpicture}[baseline= (a).base]
        \node[scale=1] (a) at (0,0){
        \begin{tikzcd}
        0 \arrow[r] & \text{H}^1(k,S)\arrow[r] & \text{H}^1(X,S)\arrow[r]& \text{Hom}_{\Gamma_k}(\widehat S, \widehat M)\arrow[r] & 0. 
       \end{tikzcd}};
        \end{tikzpicture}
\end{equation*}
Since the maps of the above exact sequence are additive, the two $S$-torsors $Y_1\longrightarrow X$ and $Y_2\longrightarrow X$ have the same type. In particular, by Lemma \ref{nonso}, the $S$-torsor $Y_2\longrightarrow X$ is a flasque quasi-resolution of $(X,x_2)$. By repeating the construction for all flasque quasi-resolutions of $(X,x_1)$, then we obtain a bijection, up to isomorphism, with flasque quasi-resolutions of $(X,x_2)$. 
\end{proof}

Now that we have completely determined the flasque quasi-resolutions of an algebraic $k$-variety with trivial $\U(X)$, we can turn our attention to the investigation of the evaluation map.

Let $Y\overset{S}{\longrightarrow} X$ be a flasque quasi-resolution of $(X,x)$. The evaluation of such a torsor in a $k$-point of $X$ defines a map as follows:
\begin{equation*}
        \begin{tikzpicture}[baseline= (a).base]
        \node[scale=1] (a) at (0,0){
        \begin{tikzcd}
        X(k) \arrow[r, "\text{ev}_x"] & \text{H}^1(k,S). 
       \end{tikzcd}};
        \end{tikzpicture}
\end{equation*}
We are going to prove that the evaluation map is well-defined, meaning that it does not depend on the choice of the flasque quasi-resolution of $(X,x)$. 

\begin{oss}
Sometimes we will simply refer to the evaluation map as `ev'. 
\end{oss}

Let $Y_1\overset{S_1}{\longrightarrow}X$ and $Y_2\overset{S_2}{\longrightarrow}X$ be two flasque quasi-resolutions of $(X,x)$. According to Proposition \ref{uni}, there are two flasque resolutions
\begin{equation*}
        \begin{tikzpicture}[baseline= (a).base]
        \node[scale=1] (a) at (0,0){
        \begin{tikzcd}
        1\arrow[r] & M\arrow[r] & S_1\arrow[r] & E_1\arrow[r] & 1
        \end{tikzcd}};
        \end{tikzpicture}
\end{equation*}
and
\begin{equation*}
        \begin{tikzpicture}[baseline= (a).base]
        \node[scale=1] (a) at (0,0){
        \begin{tikzcd}
        1\arrow[r] & M\arrow[r] & S_2\arrow[r] & E_2\arrow[r] & 1
        \end{tikzcd}};
        \end{tikzpicture}
\end{equation*}
such that the two flasque quasi-resolutions are, up to isomorphism, the pushforward of the universal $M$-torsor $Z_0 \longrightarrow M$ by $M\rightarrow S_1$ and $M\rightarrow S_2$, respectively. 

\begin{prop}
\label{triangolo}
There exists the following commutative diagram:
\begin{equation*}
        \begin{tikzpicture}[baseline= (a).base]
        \node[scale=1] (a) at (0,0){
        \begin{tikzcd}
        & \emph{H}^1(k,S_1)\arrow[dd, "\sim", bend left=75]\\
        X(k)\arrow[rd, "\emph{ev}_2"']\arrow[ru, "\emph{ev}_1"]\arrow[r, "\emph{ev}_0"']& \emph{H}^1(k,M)\arrow[d]\arrow[u]\\
        & \emph{H}^1(k,S_2).
       \end{tikzcd}};
        \end{tikzpicture}
\end{equation*}
\end{prop}

\begin{proof}
By functoriality of the pushforward, there is the following commutative diagram:
\begin{equation*}
        \begin{tikzpicture}[baseline= (a).base]
        \node[scale=1] (a) at (0,0){
        \begin{tikzcd}
        & \text{H}^1(k,S_1)\\
        X(k)\arrow[rd, "\text{ev}_2"']\arrow[r, "\text{ev}_0"']\arrow[ru, "\text{ev}_1"]& \text{H}^1(k,M)\arrow[d]\arrow[u]\\
        & \text{H}^1(k,S_2).
       \end{tikzcd}};
        \end{tikzpicture}
\end{equation*}
To conclude, we need to construct the last part of the diagram in the statement. We start with the following commutative diagram:
\begin{equation*}
        \begin{tikzpicture}[baseline= (a).base]
        \node[scale=1] (a) at (0,0){
        \begin{tikzcd}
         & 1 \arrow[d] & 1\arrow[d] \\
         1 \arrow[r]& M\arrow[d]\arrow[r] & S_1\arrow[d] \arrow[r]& E_1 \arrow[r]\arrow[d, "\id"]& 1 \\
         1\arrow[r] & S_2\arrow[d]\arrow[r] & R\arrow[d]\arrow[r] & E_1\arrow[r] & 1 \\
         & E_2\arrow[d]\arrow[r, "\id"] & E_2 \arrow[d]\\
         & 1 & 1. \\
       \end{tikzcd}};
        \end{tikzpicture}
    \end{equation*}
Since $\text{H}^1(k,E)$ and $\text{H}^1(k,E')$ are trivial, and the central exact sequences are split, we have the following diagram:
\begin{equation*}
        \begin{tikzpicture}[baseline= (a).base]
        \node[scale=1] (a) at (0,0){
        \begin{tikzcd}
         \text{H}^1(k,S_1)\arrow[rd, "\sim"]\\
         \text{H}^1(k,M)\arrow[r]\arrow[d]\arrow[u]&\text{H}^1(k,R)\\
         \text{H}^1(k,S_2).\arrow[ru, "\sim"']
       \end{tikzcd}};
        \end{tikzpicture}
\end{equation*}
By merging the two diagrams, we obtain the statement.
\end{proof}

We conclude this section by demonstrating that the evaluation map is independent of the choice of the base point.

\begin{prop}
\label{diagrammabello1}
Let $x_1$ and $x_2$ be elements in $X(k)$. Let $Y_1\overset{S_1}{\longrightarrow}X$ and $Y_2\overset{S_2}{\longrightarrow}X$ be two flasque quasi-resolutions of $(X,x_1)$ and $(X,x_2)$, respectively. There exists a commutative diagram:
\begin{equation*}
        \begin{tikzpicture}[baseline= (a).base]
        \node[scale=1] (a) at (0,0){
        \begin{tikzcd}
        & \emph{H}^1(k,S_1)\arrow[dd, "\sim"]\\
        X(k)\arrow[rd, "\eve_2"']\arrow[ru, "\eve_1"]\\
        & \emph{H}^1(k,S_2).
       \end{tikzcd}};
        \end{tikzpicture}
\end{equation*}
\end{prop}

\begin{proof}
By Proposition \ref{triangolo}, it is sufficient to reduce to the case in which the two points coincide. According to Proposition \ref{corresp}, there is a flasque quasi-resolution $Y'_1\overset{S_1}{\longrightarrow}X$ of $(X,x_2)$, which is, up to isomorphism, the image of $Y_1\overset{S_1}{\longrightarrow}X$ by $\psi$. By functoriality, we can establish the following commutative diagram:
\begin{equation*}
        \begin{tikzpicture}[baseline= (a).base]
        \node[scale=1] (a) at (0,0){
        \begin{tikzcd}
        & \text{H}^1(k,S_1)\arrow[dd, "\sim"]\\
        X(k)\arrow[rd, "\text{ev}_1"']\arrow[ru, "\text{ev}_1'"]\\
        & \text{H}^1(k,S_1).
       \end{tikzcd}};
        \end{tikzpicture}
\end{equation*}
In this diagram, the vertical arrow maps a torsor to its sum with the class of the $S$-torsor $Y_{1,x_2}\longrightarrow \text{Spec}\ k$. 
\end{proof}

\subsection{General case}
\label{altrolabel}

In this section, we are going to expand the above statements to smooth algebraic varieties that do not satisfy the condition $\text{U}(X)=0$. The idea is to reduce the proof to the previous case by embedding $X$ into a smooth compactification $X_c$. We recall that $X$ is a smooth algebraic $k$-variety such that $\pic(\overline{ X})$ is finitely generated and $X(k)$ is non-empty. The three main results of this subsection are the following. First of all, we extend the existence of flasque quasi-resolutions. Then, we show that they are, in some sense, unique. In the end, we give conditions to obtain the independence of the evaluation map by the choice of the base point and the flasque quasi-resolution.

The proof of the existence of flasque quasi-resolutions is inspired by a result of Colliot-Th\'{e}l\`{e}ne (see \cite[Theorem 5.4]{MR2404747}). 

\begin{prop}
\label{question}
There exists a flasque quasi-resolution of $(X,x)$.
\end{prop}

\begin{proof}
The algebraic $k$-variety $X_c$ satisfies the hypotheses of Section \ref{special}. Indeed $\U (X_c)$ is trivial, because $X_c$ is proper, and $\pico {X_c}$ is finitely generated, because $\pico X$ is finitely generated. By Proposition \ref{vera}, there exists a flasque quasi-resolution of $(X_c,x)$. By Proposition \ref{colliotprati}, the restriction to $X$ is a flasque quasi-resolution of $(X,x)$.  
\end{proof}

By a similar approach, we can establish the following result:

\begin{prop}
\label{esistenza2}
There exists a coflasque quasi-resolution of $X$.
\end{prop}

Now, we proceed to demonstrate the uniqueness of flasque quasi-resolutions. The statement and the proof are inspired by Colliot-Th\'{e}l\`{e}ne's work (see \cite[Proposition 3.2]{MR2404747}). 

\begin{prop}
\label{due}
Let $Y_1\overset{S_1}{\longrightarrow} X$ and $Y_2\overset{S_2}{\longrightarrow} X$ be flasque quasi-resolutions of $(X,x)$. We can establish the following properties.
\begin{itemize}
    \item There exists an isomorphism of $k$-schemes $Y_1\times_k S_2\simeq Y_2\times_k S_1$. 
    \item There are two quasi-trivial tori $P_1$ and $P_2$, such that $P_1\times_k S_2\simeq P_2\times_k S_1$ as group schemes. 
    \item There exist two morphisms $S_1\rightarrow P_1$ and $S_2\rightarrow P_2$, that induce a bijection $\emph{Ker}(\mathrm H^n(k, S_1)\rightarrow \mathrm H^n(k,P_1))\simeq \emph{Ker}(\mathrm H^n(k, S_2)\rightarrow \mathrm H^n(k,P_2))$, for all non-negative integer $n$. 
\end{itemize}
As a consequence, $\mathrm H^1(k, S_1)\simeq \mathrm H^1(k, S_2)$.
\end{prop}

\begin{proof}
    Consider the following cartesian diagram:
    \begin{equation*}
        \begin{tikzpicture}[baseline= (a).base]
        \node[scale=1] (a) at (0,0){
        \begin{tikzcd}
         Y_1\times_X Y_2 \arrow[r, "S_2"]\arrow[d, "S_1"]& Y_1 \arrow[d, "S_1"]\\
         Y_2 \arrow[r, "S_2"] & X.
       \end{tikzcd}};
        \end{tikzpicture}
\end{equation*}
According to Proposition \ref{concentrato}, torsors under flasque tori over $\mathbb G_m$-quasi-trivial algebraic $k$-varieties are constant. We observe that $(Y_1\times_X Y_2)(k)\neq \emptyset$, since there exist $y_1\in Y_1(k)$ and $y_2\in Y_2(k)$ with the same image $x\in X(k)$. This implies that the scheme $Y_1\times_X Y_2$ is isomorphic to $Y_1\times_k S_2$ (resp. $S_1\times_k Y_2$) as $S_2$-torsor (resp. $S_1$-torsor) over $Y_1$ (resp. $Y_2$). This establishes the first point. Next, we apply the exact sequence of Sansuc to the $S_2$-torsor $Y_1\times_X Y_2\longrightarrow Y_1$:
 \begin{equation*}
        \begin{tikzpicture}[baseline= (a).base]
        \node[scale=1] (a) at (0,0){
        \begin{tikzcd}
        0\arrow[r] & \U(Y_1)\arrow[r] & \U(Y_1\times_X Y_2) \arrow[r] & \widehat{S_2} \arrow[r] &  0, 
       \end{tikzcd}};
        \end{tikzpicture}
\end{equation*}
since $Y_1$ is $\mathbb G_m$-quasi-trivial. We do the same with the $S_1$-torsor $Y_1\times_X Y_2 \longrightarrow Y_2$:
\begin{equation*}
        \begin{tikzpicture}[baseline= (a).base]
        \node[scale=1] (a) at (0,0){
        \begin{tikzcd}
        0\arrow[r] & \U(Y_2)\arrow[r] & \U(Y_1\times_X Y_2) \arrow[r] & \widehat{S_1} \arrow[r] &  0.
       \end{tikzcd}};
        \end{tikzpicture}
\end{equation*}
This lead to the corresponding dual diagram of tori:
\begin{equation*}
        \begin{tikzpicture}[baseline= (a).base]
        \node[scale=1] (a) at (0,0){
        \begin{tikzcd}
        1\arrow[r] & S_2\arrow[r] & R \arrow[r]\arrow[d,"\id"] & P_1 \arrow[r] & 1 \\
        1\arrow[r] & S_1\arrow[r] & R \arrow[r] & P_2 \arrow[r] & 1,
       \end{tikzcd}};
        \end{tikzpicture}
\end{equation*}
where $P_1$ and $P_2$ are quasi-trivial tori. According to \cite[Lemma 1]{requiv}, the above exact sequences are split, which gives us the group isomorphism:
$$P_1\times_k S_2\simeq P_2\times_k S_1.$$
It remains to prove the last point. By composition, we can define the morphisms $S_2\rightarrow P_2$ and $S_1\rightarrow P_1$. Since the sequences are split, we obtain:
\begin{equation*}
        \begin{tikzpicture}[baseline= (a).base]
        \node[scale=1] (a) at (0,0){
        \begin{tikzcd}
        0\arrow[r] & \text{H}^n(k,S_2)\arrow[r] & \text{H}^n(k,R) \arrow[r]\arrow[d, "\id"] & \text{H}^n(k,P_1) \arrow[r] & 0 \\
        0\arrow[r] & \text{H}^n(k,S_1)\arrow[r] & \text{H}^n(k,R) \arrow[r] & \text{H}^n(k,P_2) \arrow[r] & 0. 
       \end{tikzcd}};
        \end{tikzpicture}
\end{equation*}
By diagram chasing, we establish the last point.
\end{proof}

\begin{oss}
\label{sonostanco}
Since the cohomology of quasi-trivial tori is zero, in the previous proposition, both items two and three imply the existence of an isomorphism 
\begin{equation*}
        \begin{tikzpicture}[baseline= (a).base]
        \node[scale=1] (a) at (0,0){
        \begin{tikzcd}
        \text{H}^1(k,S_1)\arrow[r, "\sim"] & \text{H}^1(k,S_2).
       \end{tikzcd}};
        \end{tikzpicture}
\end{equation*}
A priori they are different morphisms, but we can prove that they are, in fact, the same map. Consider the following commutative diagram:
\begin{equation*}
        \begin{tikzpicture}[baseline= (a).base]
        \node[scale=1] (a) at (0,0){
        \begin{tikzcd}
        1 \arrow[r] & S_1 \arrow[r]\arrow[d, "\id"] & S_1\times_k P_2 \arrow[r]\arrow[d, "\sim"] & P_2 \arrow[r]\arrow[d, "\id"] & 1 \\
        1 \arrow[r] & S_1 \arrow[r] & R \arrow[r]\arrow[d, "\id"] & P_2 \arrow[r] & 1 \\
        1 \arrow[r] & S_2 \arrow[r]\arrow[d, "\id"] & R \arrow[r]\arrow[d, "\sim"] & P_1 \arrow[r]\arrow[d, "\id"] & 1 \\
        1 \arrow[r] & S_2 \arrow[r] & S_2\times_k P_1 \arrow[r] & P_1 \arrow[r] & 1.
       \end{tikzcd}};
        \end{tikzpicture}
\end{equation*}
By construction, the isomorphism of the second item of Proposition \ref{due} corresponds to the vertical central composition map. 
\begin{itemize}
    \item The first isomorphism is induced by:
        \begin{equation*}
        \begin{tikzpicture}[baseline= (a).base]
        \node[scale=1] (a) at (0,0){
        \begin{tikzcd}
        S_1\arrow[r] & S_1\times_k P_2 \arrow[r, "\sim"] & R \arrow[r, "\sim"] & S_2\times_k P_1 \arrow[r] & S_2.
       \end{tikzcd}};
        \end{tikzpicture}
\end{equation*}
    \item The second isomorphism is the composition 
\begin{equation*}
        \begin{tikzpicture}[baseline= (a).base]
        \node[scale=1] (a) at (0,0){
        \begin{tikzcd}
        \text{H}^1(k,S_1)\arrow[r, "\sim"] & \text{H}^1(k,R) & \text{H}^1(k,S_2).\arrow[l, "\sim"']
       \end{tikzcd}};
        \end{tikzpicture}
\end{equation*}
\end{itemize}
By diagram chasing, we conclude that these two maps must coincide. The first interpretation is more concrete, but the second one allows us to conclude that the isomorphism does not depend on the chosen sections of the split exact sequences.
\end{oss}

\begin{oss}
Consider a flasque quasi-resolution $Y\overset{S}{\longrightarrow} X$ of $(X,x)$. The existence and the uniqueness results establish that the similarity class of $\widehat S$ is well defined and depends only on $\pico {X_c}$. Furthermore, it does not depend on the base point.
\end{oss}

For coflasque quasi-resolutions, the proposition can be formulated as follows. The argument follows the same lines as Proposition \ref{due}.

\begin{prop}
\label{unico2}
Let $Y_1\overset{P_1}{\longrightarrow} X$ and $Y_2\overset{P_2}{\longrightarrow} X$ be coflasque quasi-resolutions of $X$. We can establish the following properties.
\begin{itemize}
    \item There is an isomorphism of $k$-schemes $Y_1\times_k P_2\simeq Y_2\times_k P_1$. 
    \item Let $Q_1$ and $Q_2$ be the coflasque tori with characters group $\Ue (Y_1)$ and $\Ue (Y_2)$, respectively. There is an isomorphism $Q_1\times_k P_2\simeq Q_2\times_k P_1$ as group schemes.
    \item There are two morphisms $P_1\rightarrow Q_1$ and $S_2\rightarrow Q_2$, inducing a bijection $\emph{Coker}(\mathrm H^n(k, P_1)\rightarrow \mathrm H^n(k,Q_1))\simeq \emph{Coker}(\mathrm H^n(k, P_2)\rightarrow \mathrm H^n(k,Q_2))$, for all non-negative integer $n$. 
\end{itemize}
As a consequence, $\mathrm H^1(k, Q_1)\simeq \mathrm H^1(k, Q_2)$.
\end{prop}

We have seen in Proposition \ref{sistemaresult} and in Remark \ref{remarksenzanome}, that when $X$ is stably $k$-rational, the flasque module $\widehat S$ is stably permutation. Thanks to previous results on uniqueness and existence, we can now comprehensively describe the class of varieties for which this condition holds:

\begin{prop}
\label{hofinitoinomi}
Let $X_c$ be a smooth compactification of $X$. The following statements are equivalent:
\begin{enumerate}
    \item There exists an exact sequence:
    \begin{equation*}
        \begin{tikzpicture}[baseline= (a).base]
        \node[scale=1] (a) at (0,0){
        \begin{tikzcd}
        0 \arrow[r] & \widehat {P_2} \arrow[r] & \widehat {P_1} \arrow[r] & \emph{Pic}(\overline{X_c}) \arrow[r] & 0,
       \end{tikzcd}};
        \end{tikzpicture}
    \end{equation*}
    where $\widehat {P_i}$ are permutation lattices.
    \item If $Y\overset{S}{\longrightarrow} X$ is a flasque quasi-resolution of $(X,x)$, then $\widehat S$ is a stably permutation lattice.
\end{enumerate}
\end{prop}

\begin{proof}
$1 \Longrightarrow 2$. According to the construction outlined in Proposition \ref{question}, there exists a flasque quasi-resolution of $(X,x)$ under the torus $P_1$. Furthermore, Proposition \ref{due} establishes that the similarity class of $\widehat S$ is the same of $\widehat{P_1}$, implying that $\widehat S$ is indeed a stably permutation lattice.

$2 \Longrightarrow 1$. Let
\begin{equation*}
        \begin{tikzpicture}[baseline= (a).base]
        \node[scale=1] (a) at (0,0){
        \begin{tikzcd}
        0 \arrow[r] & \widehat P \arrow[r] & \widehat S \arrow[r] & \text{Pic}(\overline{X_c}) \arrow[r] & 0
       \end{tikzcd}};
        \end{tikzpicture}
    \end{equation*}
be a flasque resolution. By the construction of Proposition \ref{question}, there is a flasque quasi-resolution $Y\overset{S}{\longrightarrow} X$ of $(X,x)$. By hypothesis $\widehat S$ is a stably permutation lattice. Since $\widehat P$ is a permutation lattice, we can adjust it by adding an appropriate permutation lattice to both $\widehat S$ and $\widehat P$. This adjustment ensures that both $\widehat S$ and $\widehat P$ become permutation lattices, thus we get the claim.
\end{proof}

\begin{oss}
The previous proposition extends Remark \ref{remarksenzanome}. Indeed, when $X$ is a stably $k$-rational variety, the module $\pico {X_c}$ is stably permutation by \cite[Proposition 2.A.1]{sketch}. Consequently, this implies the existence of a split exact sequence as in the statement of Proposition \ref{hofinitoinomi}.
\end{oss}

The following proposition presents a result analogous to the previous one, but for coflasque quasi-resolutions:

\begin{prop}
Let $X_c$ be a smooth compactification of $X$. The following statements are equivalent:
\begin{enumerate}
    \item There exists an exact sequence:
    \begin{equation*}
        \begin{tikzpicture}[baseline= (a).base]
        \node[scale=1] (a) at (0,0){
        \begin{tikzcd}
        0 \arrow[r] & \widehat {P_2} \arrow[r] & \widehat {P_1} \arrow[r] & \emph{Pic}(\overline{X_c}) \arrow[r] & 0,
       \end{tikzcd}};
        \end{tikzpicture}
    \end{equation*}
    where $\widehat {P_i}$ are permutation lattices.
    \item If $Y\overset{P}{\longrightarrow} X$ is a coflasque quasi-resolution of $X$, then $\Ue (Y)$ is a stably permutation lattice.
\end{enumerate}
\end{prop}

\begin{proof}
$1 \Longrightarrow 2$. According to the construction outlined in Proposition \ref{question}, there exists a flasque quasi-resolution of $(X,x)$ under the torus $P_1$. By definition this is also a coflasque quasi-resolution of $X$. Furthermore, Proposition \ref{unico2} establishes that the lattice $\U (Y)$ is stably permutation. 

$2 \Longrightarrow 1$. Following the proof of Proposition \ref{esistenza2}, we consider the following cartesian diagram, where the vertical arrows are coflasque quasi-resolutions:
\begin{equation*}
        \begin{tikzpicture}[baseline= (a).base]
        \node[scale=1] (a) at (0,0){
        \begin{tikzcd}
        Y \arrow[r]\arrow[d, "P"] & Y' \arrow[d, "P"]\\
        X \arrow[r] & X_c.
       \end{tikzcd}};
        \end{tikzpicture}
    \end{equation*}
There exists the following exact sequence, since $\pico {Y'}$ is trivial:
\begin{equation*}
        \begin{tikzpicture}[baseline= (a).base]
        \node[scale=1] (a) at (0,0){
        \begin{tikzcd}
        0 \arrow[r] & \U(Y') \arrow[r] & \U(Y) \arrow[r] & \text{Div}_{\overline{Z}}\overline{Y'} \arrow[r] & 0,
       \end{tikzcd}};
        \end{tikzpicture}
    \end{equation*}
where the Galois lattice $\text{Div}_{\overline{Z}}\overline{Y'}$ is the group of divisors with support in the closed subset $\overline{Z}=\overline{Y'}\setminus \overline{Y}$. Since $\U (Y')$ is a coflasque lattice and $\text{Div}_{\overline{Z}}\overline{Y'}$ is a permutation one, there is an isomorphism $\U(Y')\oplus \text{Div}_{\overline{Z}}\overline{Y'} \simeq \U(Y)$. By hypothesis, $\U(Y)$ is a stably permutation lattice, thus $\U(Y')$ is, too. When we apply the exact sequence of Sansuc to the $P$-torsor $Y'\longrightarrow X_c$, we obtain:
\begin{equation*}
        \begin{tikzpicture}[baseline= (a).base]
        \node[scale=1] (a) at (0,0){
        \begin{tikzcd}
        0 \arrow[r] & \U (Y') \arrow[r] & \widehat P \arrow[r] & \pico {X_c} \arrow[r] & 0. 
       \end{tikzcd}};
        \end{tikzpicture}
    \end{equation*}
By adding an appropriate permutation module to both $\U (Y')$ and $\widehat P$, we establish the desired claim.
\end{proof}

In Section \ref{special}, we introduced the evaluation map induced by a flasque quasi-resolution. In that context, we demonstrated the well-defined nature of this map, meaning that it does not rely on the particular choice of the flasque quasi-resolution. Now, we seek to establish a similar result, when the algebraic $k$-variety $X$ does not satisfy the condition $\U (X)=0$. To achieve this, we present a result concerning the `versal' property of flasque quasi-resolutions, mirroring the spirit of Proposition \ref{versal}.

Consider two flasque quasi-resolutions $Y_1\overset{S_1}{\longrightarrow}X$ and $Y_2\overset{S_2}{\longrightarrow}$ of $(X,x)$. We take the following cartesian diagram:
\begin{equation*}
        \begin{tikzpicture}[baseline= (a).base]
        \node[scale=1] (a) at (0,0){
        \begin{tikzcd}
        Y_2 \times_X Y_1 \arrow[r, "S_2"]\arrow[d, "S_1"] & Y_1 \arrow[d, "S_1"]\\
        Y_2 \arrow[r, "S_2"] & X.
       \end{tikzcd}};
        \end{tikzpicture}
\end{equation*}
By construction, the set $(Y_2\times_X Y_1)(k)$ is non-empty. Therefore, by applying Proposition \ref{concentrato}, we obtain an isomorphism $ Y_2\times_X Y_1\simeq Y_2\times_k S_1$. This allows us to define the following composition map:
\begin{equation*}
        \begin{tikzpicture}[baseline= (a).base]
        \node[scale=1] (a) at (0,0){
        \begin{tikzcd}
        \widehat {S_1} \arrow[r] & \U (Y_2)\oplus \widehat {S_1} \arrow[r, "\sim"] & \U (Y_2\times_X Y_1) \arrow[r] & \widehat {S_2}.
       \end{tikzcd}};
        \end{tikzpicture}
\end{equation*}
Here, the first morphism is the canonical inclusion, and the last one is the connecting map of the exact sequence of Sansuc applied to the $S_2$-torsor $Y_2\times_X Y_1 \longrightarrow Y_1$. We denote the dual map of tori by $\chi: S_2\rightarrow S_1$. 

\begin{oss}
Even if $\chi$ depends on the isomorphism $ Y_2 \times_X Y_1\simeq Y_2\times_k S_1$, the map on cohomology $\chi_*:\Hh^1(k,S_2)\rightarrow \Hh^1(k,S_1)$ does not depend from the isomorphism, as we have seen in Remark \ref{sonostanco}. Furthermore, the morphism $\chi_*$ is bijective.
\end{oss}

The following lemma provides a more concrete interpretation of the morphism $\chi$:

\begin{lemma}
\label{forsenondovreiseparare}
We fix a point $(y_2, y_1)$ in $(Y_2\times_X Y_1)(k)$. Let $S_2 \rightarrow Y_2 \times_X Y_1$ be the embedding $s_2\rightarrow (y_2,y_1)\cdot s_2$. The composition map
\begin{equation*}
        \begin{tikzpicture}[baseline= (a).base]
        \node[scale=1] (a) at (0,0){
        \begin{tikzcd}
        S_2 \arrow[r] & Y_2\times_X Y_1 \arrow[r, "\sim"] & Y_2\times_k S_1 \arrow[r] & S_1
       \end{tikzcd}};
        \end{tikzpicture}
\end{equation*}
can be expressed as $m_{(y_2, y_1)}\circ \chi$ , where $m_{(y_2,y_1)}$ represents the multiplication by an element of $S_1(k)$.
\end{lemma}

\begin{proof}
By \cite[Lemma 6.4]{MR631309}, the induced map $\U (Y_2\times_X Y_1)\longrightarrow\widehat {S_2}$ corresponds to the connecting map of the exact sequence of Sansuc. Consequently, the map $\widehat {S_1} \longrightarrow \widehat {S_2}$ is the dual of $\chi$. By Proposition \ref{rosigene}, we can express the composition map as ${m_{(y_2,y_1)}}\circ \chi$.
\end{proof}

Now, we can establish a stronger uniqueness result about flasque quasi-resolutions:

\begin{prop}
\label{unicitaforte}
Assuming that $Y_2(k)$ is dense in $Y_2$, the flasque quasi-resolution $Y_1\overset{S_1}{\longrightarrow}X$ is, up to a sign, the pushforward of $Y_2\overset{S_2}{\longrightarrow}X$ by $\chi$. 
\end{prop}

\begin{proof}
We note that the pushforward of $Y_2\overset{S_2}{\longrightarrow}X$ by $\chi$ is the quotient of $Y_2\times_k S_1$ by the following right action of $S_2$:
\begin{equation*}
        \begin{tikzpicture}[baseline= (a).base]
        \node[scale=1] (a) at (0,0){
        \begin{tikzcd}
        ((y_2,\ s_1),\ s_2)\arrow[r] & (y_2\cdot s_2,\ \chi ({s_2})^{-1}s_1),
       \end{tikzcd}};
        \end{tikzpicture}
\end{equation*}
where the formula is on $\overline{k}$-points. It is sufficient to show that the two actions of $S_2$ on $Y_2\times_k S_1\simeq Y_2\times_X Y_1$ coincide, up to a sign. In fact, by \cite[Proposition 0.2]{GIT}, the quotient of $Y_2\times_X Y_1$ by the action of $S_2$ is $Y_1$. According to our hypothesis, the $k$-points in $Y_2\times_k S_1\times_k S_2$ are dense. Therefore, it is sufficient to prove that the two morphisms from $Y_2\times_k S_1\times_k S_2$ to $Y_2\times_k S_1$ coincide on $k$-points. We can describe the action of $S_2$ on $Y_2\times_X Y_1$ in terms of $Y_2\times_k S_1$. A point in $(Y_2\times_X Y_1)(k)$ can be presented as $(y_2\cdot s_2',\ y_1)$ for some $y_2 \in Y_2(k)$, $y_1\in Y_1(k)$ and $s_2'\in S_2(k)$. According to Lemma \ref{forsenondovreiseparare}, the image of $(y_2\cdot s_2',\ y_1)$ in $(Y_2\times_k S_1)(k)$ is $(y_2\cdot s_2',\ (m_{(y_2,y_1)}\circ \chi) (s_2'))$. The action of $S_2$ can be written in the following way:
    \begin{equation*}
        \begin{tikzpicture}[baseline= (a).base]
        \node[scale=1] (a) at (0,0){
        \begin{tikzcd}
        ((y_2\cdot s_2',\ (m_{(y_2,y_1)}\circ \chi) (s_2')),\ s_2)\arrow[r] & (y_2\cdot (s_2's_2),\ (m_{(y_2,y_1)}\circ \chi) (s_2's_2)).
       \end{tikzcd}};
        \end{tikzpicture}
    \end{equation*}
There are the following identities:
$$(m_{(y_2,y_1)}\circ \chi) (s_2's_2)=m_{(y_2,y_1)}(\chi (s_2')\ \chi(s_2))=\chi(s_2)\ m_{(y_2,y_1)}(\chi (s_2')).$$
Therefore, after composing $\chi$ with the antipodal map, the two actions coincide.     
\end{proof}

\begin{oss}
We observe that the assumption on the density of $Y_2(k)$ is satisfied when $Y_2$ is $k$-unirational. This condition can be checked on $X$, see for example the first item of Proposition \ref{sistemaresult}. In the case in which $X$ is a homogeneous space, then such a condition is satisfied. 
\end{oss}

\begin{oss}
There is a similar result for coflasque quasi-resolutions as well.
\end{oss}

\begin{esempio}
Suppose that $X$ is $\mathbb G_m$-quasi-trivial and that $X(k)$ is dense in $X$. Consider two quasi-trivial tori, denoted as $P_1$ and $P_2$. The trivial torsors $X\times_k P_1\longrightarrow X$ and $X\times_k P_2 \longrightarrow X$ are flasque quasi-resolutions. The following diagram is cartesian:
\begin{equation*}
        \begin{tikzpicture}[baseline= (a).base]
        \node[scale=1] (a) at (0,0){
        \begin{tikzcd}
        X\times_k P_1 \times_k P_2 \arrow[r]\arrow[d] & X\times_k P_2 \arrow[d]\\
        X\times_k P_1 \arrow[r] & X. 
       \end{tikzcd}};
        \end{tikzpicture}
    \end{equation*}
In this case we can take the zero map as $\chi$.
\end{esempio}

We can finally prove that, under mild assumptions, the evaluation map is independent of the choice of base point and the flasque quasi-resolution.

\begin{prop}
\label{nonsapreiproprio}
Let $x_1$ and $x_2$ be points in $X(k)$. Let $Y_1\overset{S_1}{\longrightarrow}X$ and $Y_2\overset{S_2}{\longrightarrow}X$ be flasque quasi-resolutions of $(X,x_1)$ and $(X,x_2)$, respectively. If $Y_1(k)$ and $Y_2(k)$ are dense in $Y_1$ and $Y_2$, then there is the following commutative diagram:
\begin{equation*}
        \begin{tikzpicture}[baseline= (a).base]
        \node[scale=1] (a) at (0,0){
        \begin{tikzcd}
        & \emph{H}^1(k,S_1)\arrow[dd, "\sim"]\\
        X(k)\arrow[rd, "\eve_2"']\arrow[ru, "\eve_1"]\\
        & \emph{H}^1(k,S_2).
       \end{tikzcd}};
        \end{tikzpicture}
\end{equation*}
\end{prop}

\begin{proof}
We fix a smooth compactification $X_c$ of $X$. Consider two flasque quasi-resolutions $Y_3 \overset{S_3}{\longrightarrow} X_c$ and $Y_4 \overset{S_4}{\longrightarrow} X_c$ of $(X_c,x_1)$ and $(X_c,x_2)$, respectively. According to Proposition \ref{diagrammabello1}, there is the following commutative diagram:
\begin{equation*}
        \begin{tikzpicture}[baseline= (a).base]
        \node[scale=1] (a) at (0,0){
        \begin{tikzcd}
        & \text{H}^1(k,S_3)\arrow[dd, "\sim"]\\
        X_c(k)\arrow[rd, "\ev_4"']\arrow[ru, "\ev_3"]\\
        & \text{H}^1(k,S_4).
       \end{tikzcd}};
        \end{tikzpicture}
\end{equation*}
We denote by $Y_3'\overset{S_3}{\longrightarrow} X$, the restriction of $Y_3\overset{S_3}{\longrightarrow} X_c$ to $X$. This is a flasque quasi-resolution of $(X,x_1)$ by Proposition \ref{colliotprati}, and there is the following commutative diagram:
\begin{equation*}
        \begin{tikzpicture}[baseline= (a).base]
        \node[scale=1] (a) at (0,0){
        \begin{tikzcd}
        & X_c(k)\arrow[dd, "\ev_3"]\\
        X(k)\arrow[rd, "\ev_3'"']\arrow[ru]\\
        & \text{H}^1(k,S_3).
       \end{tikzcd}};
        \end{tikzpicture}
\end{equation*}
The same applies to the torsor $Y_4\overset{S_4}{\longrightarrow} X_c$. By Proposition \ref{unicitaforte}, there exists a group morphism $\psi: S_1\rightarrow S_3$ such that $Y_3'\overset{S_3}{\longrightarrow} X$ is the pushforward of $Y_1\overset{S_1}{\longrightarrow} X$ by $\psi$. Then, by functoriality, there is the following commutative diagram: 
\begin{equation*}
        \begin{tikzpicture}[baseline= (a).base]
        \node[scale=1] (a) at (0,0){
        \begin{tikzcd}
        & \text{H}^1(k,S_1)\arrow[dd, "\sim"]\\
        X(k)\arrow[rd, "{\ev}_3'"']\arrow[ru, "\ev_1"]\\
        & \text{H}^1(k,{S}_3).
       \end{tikzcd}};
        \end{tikzpicture}
\end{equation*}
By following a similar procedure with the flasque quasi-resolutions of $(X, x_2)$, we obtain the final statement by combining all the commutative diagrams.
\end{proof}

\begin{oss}
The previous proposition applies to homogeneous spaces. 
\end{oss}

\section{Evaluation map for homogeneous spaces}
\label{homoge}

In this section, we apply the constructions obtained earlier to study the equivalence classes of a homogeneous space. We will see how the evaluation map can be used for this purpose. At the end, we will prove a slightly stronger version of a theorem of Colliot-Th\'{e}l\`{e}ne and Kunyavski\u{\i}. 

\subsection{R-equivalence on homogeneous spaces}

The notion of $R$-equivalence has been introduced by Manin, in \cite{MR833513}, for the investigation of cubic hypersurfaces. It serves as a measure, in some sense, of how close an algebraic variety is to being rational. Unfortunately, the explicit calculation of this equivalence relation has proven to be quite challenging. For instance, in his book, Manin successfully computed it only for Ch$\hat {\text{a}}$telet surfaces. 

In this section we are going to recall some well-known properties of $R$-equivalence and prove some new ones for homogeneous spaces. We refer mostly to \cite[Chapter 6]{MR1634406}.

\begin{Def}
\label{requivalenza}
Let $X$ be an algebraic $k$-variety, with $X(k)$ non-empty. Let $x,y \in X(k)$. We say that $x$ and $y$ are strictly $R$-equivalent, and we write $x\underset{SR}{\sim}y $, if there is a morphism
\begin{equation*}
        \begin{tikzpicture}[baseline= (a).base]
        \node[scale=1] (a) at (0,0){
        \begin{tikzcd}
       U\arrow[r,"\phi"] & X, 
       \end{tikzcd}};
        \end{tikzpicture}
\end{equation*}
such that $U$ is an open subset of $\mathbb A^1_k$ and $x,y\in \phi\left( U(k)\right)$. If there exists a sequence of points $x_0=x$, $x_1$, ... , $x_{n+1}=y$ in $X(k)$ such that for all $i=0,\cdots, n$ the points $x_i$ and $x_{i+1}$ are strictly $R$-equivalent, then we say that $x$ and $y$ are $R$-equivalent and we write $x\underset{R}{\sim}y $. We denote the set of $R$-equivalence classes by $X(k)/R$. 
\end{Def}

\begin{oss}
There is also a more functorial definition of $R$-equivalence. See \cite[Definition 1.6.1]{MR3972198} for more details.
\end{oss}

The following proposition follows directly from the definition:

\begin{prop}
Let $X$ and $Y$ be $k$-algebraic varieties. 
\begin{itemize}
    \item Any $k$-morphism $X\rightarrow Y$ induces a map $X(k)/R\rightarrow Y(k)/R$.
    \item There is a canonical bijection $(X\times_k Y)(k)/R=X(k)/R \times Y(k)/R$.
    \item Let $l\subset k$ be a finite separable extension. If $R_{k/l}(X)$ exists as an algebraic $l$-variety, then there is a canonical bijection $R_{k/l}(X)(l)/R=X(k)/R$.
\end{itemize}
\end{prop}

\begin{oss}
We can deduce that for all quasi-trivial tori $T$, the group $T(k)/R$ is trivial.
\end{oss}

$R$-equivalence is a birational invariant for homogeneous spaces. Indeed, we can prove the following result, that is a slight generalisation of \cite[Proposition 11]{requiv}: 

\begin{prop}
\label{birahom}
Let $X$ be a homogeneous space under a connected linear $k$-algebraic group $G$. Let $U$ be a non-empty open subset of $X$. If $X(k)$ is non-empty, then $U(k)$ is non-empty and the canonical map
$$U(k)/R\longrightarrow X(k)/R$$
is a bijection.
\end{prop}

\begin{proof}
Since the field $k$ has characteristic zero, there is a dominant morphism $W \rightarrow G$, where $W$ is an open subset of $\mathbb P^n_k$. Up to multiplication by an element of $G(k)$, we can suppose that the identity of $G(k)$ is the image of a rational point of $W$.
\begin{itemize}
    \item The map is surjective. Let $x\in X(k)$. We have to prove that the $R$-equivalence class of $x$ intersects $U$. Let $G\rightarrow X$ be the surjective morphism $g\rightarrow g\cdot x$. The image of $W(k)$ in $X$ is a dense subset that is contained in the $R$-equivalence class of $x$. That proves the claim.
    \item The map is injective. Let $x, y \in U(k)$. We assume the existence of a chain of elements $x=x_0, x_1, \cdots , x_n, x_{n+1}=y$ in $X(k)$ such that $x_i$ and $x_{i+1}$ are strictly $R$-equivalent for all $i=0,\cdots, n$. For all $i=0,\cdots, n+1$ we can define the morphism $\phi_i:G\rightarrow X$ that is $g\rightarrow g\cdot x_i$. Let $V$ be the intersection of all open subset $\phi_i^{-1}(U)$. The image of $W(k)$ in $G$ is dense, in particular it intersects $V$, since it is non-empty. Let $g$ be an element in the intersection. By construction, there is a chain of elements of $U$: $x\underset{SR}{\sim} g\cdot x \underset{SR}{\sim} g \cdot x_1 \underset{SR}{\sim} \cdots \underset{SR}{\sim} g\cdot x_n \underset{SR}{\sim} g \cdot y \underset{SR}{\sim} y $.
    \end{itemize}
\end{proof}

There is the following corollary:

\begin{cor}
\label{cormoving}
Let $X$ and $Y$ be homogeneous spaces under connected algebraic $k$-groups, with $X(k)$ and $Y(k)$ non-empty. If $X$ and $Y$ are birationally equivalent, then there is a bijective map between $X(k)/R$ and $Y(k)/R$.   
\end{cor}

We recall the following definition:

\begin{Def}
Let $Y\longrightarrow X$ be a torsor under a linear algebraic $k$-group. We say that the torsor is generically trivial, if there exists a non-empty open subset $U$ of $X$ such that the restriction to $U$ is isomorphic to the trivial torsor $U\times_k G\longrightarrow U$. 
\end{Def}

In the next proposition, we compute $R$-equivalence in a specific case of homogeneous space:

\begin{prop}
\label{semplice}
Let $X=G/M$ be a homogeneous space under a connected linear algebraic $k$-group $G$. We suppose that $M$ is a group of multiplicative type, and that every $G$-torsor is generically trivial in the Zariski topology. Let 
\begin{equation*}
        \begin{tikzpicture}[baseline= (a).base]
        \node[scale=1] (a) at (0,0){
        \begin{tikzcd}
       1\arrow[r] & M \arrow[r] & S \arrow[r] & E \arrow[r] & 1
       \end{tikzcd}};
        \end{tikzpicture}
\end{equation*}
be a flasque resolution of $M$. Then, there exists a bijection:
$$X(k)/R \simeq \emph{H}^1(k, S)\times G(k)/R .$$
\end{prop}

\begin{proof}
Let 
\begin{equation*}
        \begin{tikzpicture}[baseline= (a).base]
        \node[scale=1] (a) at (0,0){
        \begin{tikzcd}
       1\arrow[r] & M \arrow[r] & P \arrow[r] & Q \arrow[r] & 1
       \end{tikzcd}};
        \end{tikzpicture}
\end{equation*}
be a coflasque resolution. We consider the $G$-torsors $(G\times_k P)/M\longrightarrow Q$ and the $P$-torsor $(G\times_k P)/M\longrightarrow X$. Both torsors are generically trivial, leading to the birational isomorphism:
$$X\times_k P\underset{\text{bir}}{\sim}Q\times_k G.$$
By Corollary \ref{cormoving}, $R$-equivalence is a birational invariant for homogeneous spaces:
$$X(k)/R\times P(k)/R\simeq Q(k)/R\times G(k)/R.$$
The group $P(k)/R$ is trivial because $P$ is a quasi-trivial torus. Furthermore, the exact sequence 
\begin{equation*}
        \begin{tikzpicture}[baseline= (a).base]
        \node[scale=1] (a) at (0,0){
        \begin{tikzcd}
       1\arrow[r] & S \arrow[r] & (S\times_k P)/M \arrow[r] & Q \arrow[r] & 1
       \end{tikzcd}};
        \end{tikzpicture}
\end{equation*}
is a flasque resolution of $Q$, yielding $Q(k)/R\simeq \text{H}^1(k,S)$. 
\end{proof}

\begin{oss}
In the upcoming sections, we will provide a more comprehensive interpretation of the bijection established in the previous proposition.
\end{oss}

\begin{oss}
The previous proposition applies to various scenarios, such as when $G$ is equal to $GL_n$, $SL_n$, or $Sp_{2n}$ because every torsor under one of these groups is trivialized by a Zariski covering.
\end{oss}

\subsection{Second construction of flasque quasi-resolutions}
\label{secondco}

Since the evaluation map does not depend on the choice of resolution, we can choose a more suitable one. In this section, we provide a more specific construction of flasque quasi-resolution.

\begin{oss}
For instance, in the case in which $H$ is connected, for the existence of flasque quasi-resolutions of $X$, there is no need to take a flasque resolution of $\pico {X_c}$, because this module is already flasque by \cite[Theorem 5.1]{MR2237268}. The restriction of the universal torsor over $X_c$, that is trivial over $x_0$, is a flasque quasi-resolution of $(X,x_0)$.
\end{oss}

We recall the following known result:

\begin{prop}
\label{pranzo}
Let $G$ be a linear algebraic $k$-group.
\begin{itemize}
    \item There exists $M\coloneqq G^{\emph{mult}}$, which is the largest quotient of multiplicative type of $G$.
    \item If $k\subset k'$ is a Galois extension, then $M_{k'}$ is the largest quotient of multiplicative type of $N_{k'}$.
    \item The canonical map $\widehat M \rightarrow \widehat G$ is bijective.
\end{itemize}
\end{prop}

\begin{proof}
Given a field extension $k\subset \Tilde{k}$, we denote by $\mathfrak F(\Tilde{k})$ the set of normal subgroups $G'$ of $G_{\Tilde{k}}$ such that $G_{\Tilde{k}}/G'$ is a group of multiplicative type.
\begin{itemize}
    \item The set $\mathfrak F(\Tilde{k})$ has a unique minimal element $G^{\text{min}}_{\Tilde{k}}$. Indeed, we can assume that $k=\Tilde{k}$, without loss of generality. There is a minimal element because $G$ is a noetherian topological space. We assume that $G'$ and $G''$ are two minimal elements of $\mathfrak F(k)$. The intersection $G'\cap G''$ is a normal subgroup of $G$ and there is a inclusion:
    \begin{equation*}
        \begin{tikzpicture}[baseline= (a).base]
        \node[scale=1] (a) at (0,0){
        \begin{tikzcd}
        G/(G'\cap G'') \arrow[r] & G/G' \times G/G''.
       \end{tikzcd}};
        \end{tikzpicture}
    \end{equation*}
    Since $G/G'$ and $G/G''$ are of multiplicative type, the group $G/(G'\cap G'')$ is of multiplicative type. Thus, by minimality, $G'=G'\cap G''=G''$.
    \item Let $\Gamma$ be the Galois group of the extension $k\subset k'$. By minimality, the action of $\Gamma$ on $G_{k'}$ preserves the group $G^{\text{min}}_{k'}$. By applying Galois descent, there exists a subgroup $G'$ of $G$ such that $G'_{k'}= G^{\text{min}}_{k'}$. The group $G'$ belongs to $\mathfrak F (k)$, since $G^{\text{min}}_{k'}$ is an element of $\mathfrak F(k')$. There is an inclusion $ G^{\text{min}}_{k'} \subset ( G^{\text{min}}_{k})_{k'}$ and thus $G' \subset G^{\text{min}}_{k}$. By minimality of $G^{\text{min}}_{k}$, these groups must be identical.
    \item The map is injective, since $M$ is a quotient of $G$. To establish the surjectivity we need to demonstrate that for all group morphisms $G_{\overline{k}}\rightarrow \mathbb G_{m,\overline{k}}$, there exists a commutative diagram:
    \begin{equation*}
        \begin{tikzpicture}[baseline= (a).base]
        \node[scale=1] (a) at (0,0){
        \begin{tikzcd}
        & M_{\overline{k}}\arrow[d, dashed]\\
        G_{\overline{k}}\arrow[r]\arrow[ur]& \mathbb G_{m, \overline{k}}.
       \end{tikzcd}};
        \end{tikzpicture}
\end{equation*}    
    By the previous item, the group $M_{\overline{k}}$ is the largest quotient of multiplicative type of $G_{\overline{k}}$ and thus we get the statement.
\end{itemize}
\end{proof}

We consider a homogeneous space $X=G/H$, where $G$ is a $\mathbb G_m$-quasi-trivial linear algebraic $k$-group, and $H$ is a closed subgroup. We denote by $x_0\in X(k)$ the class of the identity of $G$. 

\begin{oss}
By \cite[Proposition 3.1]{MR2404747}, every homogeneous space under a connected linear algebraic $k$-group is a homogeneous space under a $\mathbb G_m$-quasi-trivial linear algebraic $k$-group.
\end{oss}

We proceed with the construction of a flasque quasi-resolution of $(X,x_0)$. Let $M$ be the maximal quotient of multiplicative type of $H$. Let
\begin{equation*}
        \begin{tikzpicture}[baseline= (a).base]
        \node[scale=1] (a) at (0,0){
        \begin{tikzcd}
        1\arrow[r] & M\arrow[r] & S_0 \arrow[r] & E_0 \arrow[r] & 1
       \end{tikzcd}};
        \end{tikzpicture}
\end{equation*}
be a flasque resolution of $M$. We take the pushforward of the $H$-torsor $G\longrightarrow X$ by the group morphism $H\rightarrow M \rightarrow S_0$:
\begin{equation*}
        \begin{tikzpicture}[baseline= (a).base]
        \node[scale=1] (a) at (0,0){
        \begin{tikzcd}
         G \arrow[r]\arrow[d, "H"]& Y_0\coloneqq (G\times_k S_0) /H \arrow[ld, "S_0"]\\
         X.
       \end{tikzcd}};
        \end{tikzpicture}
\end{equation*}

We need the following lemma:

\begin{lemma}
\label{risolvetutto}
Let $G$ be a connected linear algebraic $k$-group, with trivial $\emph{Pic}(\overline {G})$, and let $X=G/H$ be a homogeneous space. There are the following isomorphisms of $\Gamma_k$-modules:
\begin{itemize}
    \item $\Ue (X)\simeq \emph{Ker} (\widehat {G} \rightarrow \widehat {H})$;
    \item $\emph{Pic} (\overline{X}) \simeq \emph {Coker}(\widehat {G} \rightarrow \widehat {H})$.
\end{itemize}
\end{lemma}

\begin{proof}
The proof of \cite[Proposition 2.5]{MR2952422} works under our assumptions.
\end{proof}

We can prove:

\begin{prop}
\label{oralanomino}
The $S_0$-torsor $Y_0\longrightarrow X$ is a flasque quasi-resolution of $(X,x_0)$. 
\end{prop}

\begin{proof}
We need to demonstrate that $Y_0$ is $\mathbb G_m$-quasi-trivial. By construction, the embedding of $H$ into $G\times_k S_0$ induces a surjective morphism of $\Gamma_k$-modules $\widehat G \oplus \widehat {S_0} \rightarrow \widehat H\simeq \widehat M$. By Lemma \ref{risolvetutto}, the group $\text{Pic}(\overline {Y_0})$ is trivial and there is an isomorphism $\text{U}(Y_0)\simeq \text{Ker}(\widehat G\oplus \widehat {S_0} \rightarrow \widehat H)$. There is the following commutative diagram with exact rows:
    \begin{equation*}
        \begin{tikzpicture}[baseline= (a).base]
        \node[scale=1] (a) at (0,0){
        \begin{tikzcd}
        0 \arrow[r]& \widehat {E_0} \arrow[r]\arrow[d, "\id"]&  \text{U}(Y_0)\arrow[r]\arrow[d]& \widehat G \arrow[r]\arrow[d]& 0 \\
        0 \arrow[r]& \widehat {E_0} \arrow[r]& \widehat {S_0}\arrow[r] & \widehat H\arrow[r] & 0.
       \end{tikzcd}};
        \end{tikzpicture}
    \end{equation*}     
By \cite[Lemma 1]{requiv}, the group $\U(Y_0)$ is a permutation module, since $\widehat {E_0}$ and $\widehat G$ are. We have shown that the $S_0$-torsor $Y_0\longrightarrow X$ is a flasque quasi-resolution of $(X,x_0)$. 
\end{proof}

Similarly, we can consider a coflasque resolution of $M$:
\begin{equation*}
        \begin{tikzpicture}[baseline= (a).base]
        \node[scale=1] (a) at (0,0){
        \begin{tikzcd}
        1\arrow[r] & M\arrow[r] & P_0 \arrow[r] & Q_0 \arrow[r] & 1.
       \end{tikzcd}};
        \end{tikzpicture}
\end{equation*}
In the same way of Proposition \ref{oralanomino}, we can prove the following:

\begin{prop}
The pushforward of the $H$-torsor $G\longrightarrow X$ by the group morphism $H\rightarrow M \rightarrow P_0$ is a coflasque quasi-resolution of $X$.
\end{prop}

\begin{oss}
It is not a coincidence that for the flasque quasi-resolution $Y_0\overset{S_0}{\longrightarrow}X$ of $(X,x_0)$ the algebraic $k$-variety $Y_0$ has a homogeneous space structure. In fact, there is a result of Cao, that explains such a phenomenon (see \cite[Theorem 2.7]{MR3778194}).
\end{oss}

\subsection{General properties of the evaluation map}

In this subsection, we would like to investigate injectivity and/or surjectivity of the evaluation map. Let $X=G/H$ be a homogeneous space under a connected linear algebraic $k$-group $G$. We fix a point $x\in X(k)$ and a flasque quasi-resolution $Y\overset{S}{\longrightarrow}X$ of $(X,x)$. We can define an evaluation map:
\begin{equation*}
        \begin{tikzpicture}[baseline= (a).base]
        \node[scale=1] (a) at (0,0){
        \begin{tikzcd}
        X(k)\arrow[r, "\ev"] & \text{H}^1(k,S).
       \end{tikzcd}};
        \end{tikzpicture}
\end{equation*}

In general, the evaluation map is not injective, indeed we can prove the following known result:

\begin{prop}
\label{requiv}
The evaluation map factors through $R$-equivalence:
\begin{equation*}
        \begin{tikzpicture}[baseline= (a).base]
        \node[scale=1] (a) at (0,0){
        \begin{tikzcd}
         X(k)/R\arrow[r, "\eve"] & \emph{H}^1(k,S).
       \end{tikzcd}};
        \end{tikzpicture}
\end{equation*}
\end{prop}

\begin{proof}
It is sufficient to establish that, given two points $x_0, x_1 \in X(k)$ that are strictly $R$-equivalent, there is the identity $\text{ev}(x_0)=\text{ev}(x_1)$. By definition, there is an open subset $U$ of $\mathbb A^1_k$ and a $k$-morphism $\phi:U\rightarrow X$ such that $\phi \left(U(k)\right)\ni x_0, x_1$. Consider the following cartesian diagram:
\begin{equation*}
        \begin{tikzpicture}[baseline= (a).base]
        \node[scale=1] (a) at (0,0){
        \begin{tikzcd}
         Y_U \arrow[d, "S"]\arrow[r]& Y\arrow[d, "S"]\\
         U\arrow[r, "\phi"] & X.
       \end{tikzcd}};
        \end{tikzpicture}
\end{equation*}
By \cite[Corollary 2.6]{MR878473} the torsor $Y_U \overset{S}{\longrightarrow} U$ is constant. Therefore, we obtain the identity $\text{ev}(x_0)=\text{ev}(x_1)$.
\end{proof}

\begin{oss}
From now on, when we say evaluation map, we could assume that the domain is $X(k)/R$. 
\end{oss}

For flasque resolutions of connected linear algebraic $k$-groups, the evaluation map is a group morphism. In particular, the map has trivial kernel if and only if it is injective. However, in the case of homogeneous spaces, this is no longer true. Instead, we have the following weaker result:

\begin{prop}
\label{nonmoltointeressante}
The evaluation map 
\begin{equation*}
        \begin{tikzpicture}[baseline= (a).base]
        \node[scale=1] (a) at (0,0){
        \begin{tikzcd}
         X(k)/R\arrow[r, "\eve"] & \emph{H}^1(k,S)
       \end{tikzcd}};
        \end{tikzpicture}
\end{equation*}
is injective if and only if, for all $x_0\in X(k)$, the evaluation map
\begin{equation*}
        \begin{tikzpicture}[baseline= (a).base]
        \node[scale=1] (a) at (0,0){
        \begin{tikzcd}
         X(k)/R\arrow[r, "\eve_0"] & \emph{H}^1(k,S_0),
       \end{tikzcd}};
        \end{tikzpicture}
\end{equation*}
induced by a flasque quasi-resolution $Y_0\overset{S_0}{\longrightarrow}X$ of $(X,x_0)$, has trivial kernel.
\end{prop}

\begin{proof}
By Proposition \ref{nonsapreiproprio}, we have the following commutative diagram:
\begin{equation*}
        \begin{tikzpicture}[baseline= (a).base]
        \node[scale=1] (a) at (0,0){
        \begin{tikzcd}
        & \text{H}^1(k,S_0)\arrow[dd, "\sim"]\\
        X(k)/R\arrow[rd, "\ev"']\arrow[ru, "\ev_0"]\\
        & \text{H}^1(k,S).
       \end{tikzcd}};
        \end{tikzpicture}
\end{equation*}
This implies that if $\text{ev}$ is injective, then $\ev_0$ is injective as well. Now, suppose by contradiction that there exist $x_1$ and $x_2$ in $X(k)$ such that $\ev(x_1)=\ev(x_2)$, but they are not $R$-equivalent. We can consider a flasque quasi-resolution of $(X,x_1)$. Thanks to the above diagram, we obtain that $\ev(x_1)=\ev(x_2)$. By hypothesis, the map $\ev_0$ has trivial kernel, thus $x_1$ and $x_2$ are $R$-equivalent, which is a contradiction.
\end{proof}

We can give an example in which the injectivity is verified:

\begin{prop}
\label{esempietto}
Suppose that $G$ is $\mathbb G_m$-quasi-trivial and that $H$ is connected and reductive. We assume the following properties:
\begin{itemize}
    \item the group $G(k)/R$ is trivial,
    \item the set $\Hh^1(k,\Tilde{H}^{\emph{ss}})$ is trivial, for all forms $\Tilde{H}^{\emph{ss}}$ of the semi-simple part of $H$,
    \item the maximal toric quotient $H^{\emph{tor}}$ of $H$ is flasque.
\end{itemize}
Then, the evaluation map is injective.
\end{prop}

\begin{proof}
Let $x_0$ be a point in $X(k)$. The stabiliser of $x_0$ is a form $\Tilde{H}$ of $H$ that fits in the following central exact sequence:
\begin{equation*}
        \begin{tikzpicture}[baseline= (a).base]
        \node[scale=1] (a) at (0,0){
        \begin{tikzcd}
        1 \arrow[r] & \Tilde{H}^{\text{ss}}\arrow[r] & \Tilde{H}\arrow[r] & H^{\text{tor}}\arrow[r] & 1.
       \end{tikzcd}};
        \end{tikzpicture}
\end{equation*}
Due to Proposition \ref{oralanomino}, the pushforward of $G\rightarrow G/\Tilde{H}=X$ by the map $\Tilde{H}\rightarrow H^{\text{tor}}$ is a flasque quasi-resolution of $(X,x_0)$. Therefore, the induced evaluation map $\ev$ fits into the following commutative diagram:
\begin{equation*}
        \begin{tikzpicture}[baseline= (a).base]
        \node[scale=1] (a) at (0,0){
        \begin{tikzcd}
        & G(k)\arrow[d] \\
        & X(k)\arrow[d]\arrow[rd, "\ev"] \\
        \Hh^1(k, \Tilde{H}^{\text{ss}})\arrow[r]& \Hh^1(k, \Tilde{H})\arrow[r] & \Hh^1(k, H^{\text{tor}}).
        \end{tikzcd}};
        \end{tikzpicture}
\end{equation*}
By hypothesis, the set $\Hh^1(k, \Tilde{H}^{\text{ss}})$ is trivial and thus the map $\Hh^1(k,\Tilde{H})\rightarrow \Hh^1(k, H^{\text{tor}})$ has trivial kernel. This implies that if $\Tilde{x}$ is in the kernel of the evaluation map, then it is in the kernel of $X(k)\rightarrow \Hh^1(k,\Tilde{H})$. Therefore, $\Tilde{x}$ is in the image of $G(k)\rightarrow X(k)$. By hypothesis, the group $G(k)/R$ is trivial, thus $\Tilde{x}$ is $R$-equivalent to the class of the identity of $G$. We have just demonstrated that the map
\begin{equation*}
        \begin{tikzpicture}[baseline= (a).base]
        \node[scale=1] (a) at (0,0){
        \begin{tikzcd}
        X(k)/R \arrow[r, "\ev"] & \text{H}^1(k,H^{\text{tor}})
       \end{tikzcd}};
        \end{tikzpicture}
\end{equation*}
has trivial kernel. We conclude by applying Proposition \ref{nonmoltointeressante}. 
\end{proof}

\begin{oss}
For example, Proposition \ref{esempietto} applies when the field $k$ is good and $H$ is a central extension of a flasque torus by a semi-simple simply connected group. 
\end{oss}

We are going to prove that, the bijection described in Proposition \ref{semplice}, it is induced by a flasque quasi-resolution, when the group $G$ is $\mathbb G_m$-quasi-trivial:

\begin{prop}
We assume that $G$ is $\mathbb G_m$-quasi-trivial and that all the $G$-torsors are generically trivial. We assume that $H=M$ is a group of multiplicative type. The evaluation map is surjective, and it is injective if and only if $G(k)/R$ is trivial.
\end{prop}

\begin{proof}
Let
\begin{equation*}
        \begin{tikzpicture}[baseline= (a).base]
        \node[scale=1] (a) at (0,0){
        \begin{tikzcd}
       1\arrow[r] & M \arrow[r] & P \arrow[r] & Q \arrow[r] & 1,
       \end{tikzcd}};
        \end{tikzpicture}
\end{equation*}
be a coflasque resolution. We consider the $G$-torsor $(G\times_k P)/M\longrightarrow Q$ and the $P$-torsor $(G\times_k P)/M\longrightarrow  X$. Due to Proposition \ref{colliotprati}, the pull-back by these torsors preserves flasque quasi-resolutions. Then, the second torsor induces the following commutative diagram: 
\begin{equation*}
        \begin{tikzpicture}[baseline= (a).base]
        \node[scale=1] (a) at (0,0){
        \begin{tikzcd}
       ((G\times_k P)/M)(k)/R \arrow[rd, "\ev"]\arrow[dd, "\sim"] \\
       & \text{H}^1(k,S).\\
       X(k)/R\arrow[ur, "\ev"]
       \end{tikzcd}};
        \end{tikzpicture}
\end{equation*}
Let $U$ be an open subset of $X$, that makes the torsor $(G\times_k P)/M\overset{G}{\longrightarrow} Q$ trivial. There is the following commutative diagram:
\begin{equation*}
        \begin{tikzpicture}[baseline= (a).base]
        \node[scale=1] (a) at (0,0){
        \begin{tikzcd}
       U(k)/R \times G(k)/R\arrow[r, "\sim"]\arrow[dd]&((G\times_k P)/M)(k)/R \arrow[rd, "\ev"]\arrow[dd] \\
       && \text{H}^1(k,S)\\
       U(k)/R\arrow[r, "\sim"]&Q(k)/R\arrow[ur, "\ev", "\sim"']
       \end{tikzcd}};
        \end{tikzpicture}
\end{equation*}
We obtain the claim by combining the two diagrams.
\end{proof}

\subsection{On a theorem of Colliot-Th\'{e}l\`{e}ne and Kunyavski\u{\i}}

Let $X=G/H$ be a homogeneous space under a connected linear algebraic $k$-group $G$ with connected stabiliser $H$, and let $X_c$ be a smooth compactification of $X$. We denote by $x_0$ the class of the identity of $G$. Let $Y_0\longrightarrow X_c$ be the universal $S_0$-torsor with trivial fibre over $x_0$. 

We recall the following definition:

\begin{Def}
\label{goodfield}
Let $k$ be a field of cohomological dimension $\leq 2$. We say that $k$ is good if it satisfies the following properties for all finite extensions $k\subset k'$:
\begin{itemize}
    \item index and exponent of all central simple algebras over $k'$ coincide;
    \item for all $\mathbb G_m$-quasi-trivial linear algebraic $k'$-group $\Tilde G$, the set $\text{H}^1(k', \Tilde G)$ is trivial.
\end{itemize}
\end{Def}

\begin{prop}
\label{excoro}
Let $X'$ be a homogeneous space under a connected linear algebraic $k$-group $G'$. Suppose that the maximal toric quotient of the geometric stabiliser is trivial. Suppose that the field $k$ is good.
\begin{itemize}
    \item There exists a $G$-torsor that dominates $X'$.
    \item If $G'$ is $\mathbb G_m$-quasi-trivial, then $X'$ has a $k$-point.
\end{itemize}
\end{prop}

\begin{proof}
The first item is a combination of \cite[Remark 4.1.1]{MR2237268} and \cite[Section 5.1]{MR1608617}. If $G'$ is $\mathbb G_m$-quasi-trivial, then every $G'$-torsor is trivial, since $k$ is good. Thus, there is a map $G'\rightarrow X'$, in particular $X'(k)$ is non-empty.
\end{proof}

In \cite[Theorem 6.1]{MR2237268}, Colliot-Th\'{e}l\`{e}ne and Kunyavski\u{\i} proved the following theorem:

\begin{thm}
\label{ctkoriginal}
The torus $S_0$ is flasque and the evaluation map induced by such a universal torsor is surjective if $k$ is good.
\end{thm}

\begin{oss}
The universal torsor is a flasque quasi-resolution of $(X_c, x_0)$, thus the surjectivity is equivalent to prove that the evaluation map induced by $Y_0\longrightarrow X_c$ a flasque quasi-resolution of $(X_c,x_0)$ is surjective. 
\end{oss}

\begin{oss}
As a corollary, the theorem gives a lower bound to the number of classes of $R$-equivalence of $X_c(k)$.
\end{oss}

In this section, we want to prove a slightly stronger result of the one above:

\begin{thm}
\label{ctkthm}
The evaluation map induced by a flasque quasi-resolution of $(X,x_0)$ is surjective if the field $k$ is good.
\end{thm}

The idea of the proof is similar to the one of Colliot-Th\'{e}l\`{e}ne and Kunyavski\u{\i}, but we are going to prove it with the language of flasque quasi-resolutions. We are going to give two proof of this theorem, we start with the longer one in which we can see how to play with flasque quasi-resolutions.

Let $H^{\text{tor}}$ be the maximal toric quotient of $H$. The kernel $H'$ of $H\rightarrow H^{\text{tor}}$ is a connected linear algebraic $k$-group, whose maximal toric quotient is trivial (see Theorem 3.1 and Theorem 5.1 of \cite[Chapter XVII]{milneAGS}). Let \begin{equation*}
        \begin{tikzpicture}[baseline= (a).base]
        \node[scale=1] (a) at (0,0){
        \begin{tikzcd}
        1\arrow[r] & H^{\text{tor}} \arrow[r]& P_0\arrow[r] & Q_0\arrow[r] & 1
       \end{tikzcd}};
        \end{tikzpicture}
\end{equation*}
be a coflasque resolution of $H^{\text{tor}}$. Let 
\begin{equation*}
        \begin{tikzpicture}[baseline= (a).base]
        \node[scale=1] (a) at (0,0){
        \begin{tikzcd}
        1\arrow[r] & S_0 \arrow[r]& E_0\arrow[r] & Q_0\arrow[r] & 1
       \end{tikzcd}};
        \end{tikzpicture}
\end{equation*}
be a flasque resolution of $Q_0$. We take the following commutative diagram:
\begin{equation*}
        \begin{tikzpicture}[baseline= (a).base]
        \node[scale=1] (a) at (0,0){
        \begin{tikzcd}
        && 1\arrow[d] & 1\arrow[d]\\
        && S_0\arrow[d]\arrow[r, "\id"] & S_0\arrow[d]\\
        1\arrow[r] & H^{\text{tor}}\arrow[d, "\id"]\arrow[r] & F_0\arrow[d]\arrow[r] & E_0\arrow[d]\arrow[r] & 1\\
        1\arrow[r] & H^{\text{tor}} \arrow[r]& P_0\arrow[d]\arrow[r] & Q_0\arrow[d]\arrow[r] & 1\\
        && 1 & 1
       \end{tikzcd}};
        \end{tikzpicture}
\end{equation*}

Let $Z_0\coloneqq (G\times_k P_0)/H\rightarrow X$ be the pushforward of $G\rightarrow X$ by $H\rightarrow H^{\text{tor}} \rightarrow P_0$. 
The canonical projection $Z_0\rightarrow  Q_0$ has a structure of $G$-homogeneous space. We denote by $z_0$ the class of the identity of $G\times_k P_0$. 

\begin{prop}
\label{nuova}
The evaluation map $\emph{ev}:X(k)\rightarrow \emph{H}^1(k,S)$ is surjective if and only if the evaluation map associated to a flasque quasi-resolution of $(Z_0,z_0)$ is surjective.
\end{prop}

\begin{proof}
By Proposition \ref{nonsapreiproprio}, it is enough to prove the claim for a fixed flasque quasi-resolution of $(Z_0, z_0)$. We take the following cartesian diagram:
\begin{equation*}
        \begin{tikzpicture}[baseline= (a).base]
        \node[scale=1] (a) at (0,0){
        \begin{tikzcd}
       Y'\arrow[r]\arrow[d, "S"] & Y \arrow[d, "S"]\\
       Z_0 \arrow[r, "P_0"] & X.
       \end{tikzcd}};
        \end{tikzpicture}
\end{equation*}
By Proposition \ref{colliotprati}, the $S$-torsor $Y'\longrightarrow Z_0$ is a flasque quasi-resolution of $(Z_0,z_0)$. There is the following commutative diagram:
\begin{equation*}
        \begin{tikzpicture}[baseline= (a).base]
        \node[scale=1] (a) at (0,0){
        \begin{tikzcd}
       Z_0(k)\arrow[d]\arrow[r, "\ev_{Z_0}"] & \text{H}^1(k,S) \arrow[d, "\id"]\\
       X(k)\arrow[r, "\ev"] & \text{H}^1(k,S).
       \end{tikzcd}};
        \end{tikzpicture}
\end{equation*}
Since $Z_0 \longrightarrow X$ is a torsor under a quasi-trivial torus, the map $Z_0(k)\rightarrow X(k)$ is surjective and thus $\text{ev}_X$ is surjective if and only if $\text{ev}_{Z_0}$ is.
\end{proof}

\begin{oss}
The map $Z_0(k)\rightarrow X(k)$ induces a bijective map $Z_0(k)/R \rightarrow X(k)/R$. Thus a similar statement holds for injectivity, too.
\end{oss}

\begin{prop}
\label{ctk3}
Suppose that $G$ is $\mathbb G_m$-quasi-trivial. The pull-back of the $S_0$-torsor $E_0\longrightarrow Q_0$ by the $G$-homogeneous space $Z_0\longrightarrow Q_0$ is a flasque quasi-resolution of $(Z_0,z_0)$.
\end{prop}

\begin{proof}
First of all, we show that the following diagram is cartesian:
\begin{equation*}
        \begin{tikzpicture}[baseline= (a).base]
        \node[scale=1] (a) at (0,0){
        \begin{tikzcd}
       (G\times_k F_0)/H \arrow[r]\arrow[d]& E_0 \arrow[d]\\
       Z_0\arrow[r] & Q_0,
       \end{tikzcd}};
        \end{tikzpicture}
\end{equation*}
where $(G\times_k F_0)/H \rightarrow E_0$ and $(G\times_k F_0)/H\rightarrow Z_0$ are the canonical maps. We prove that
\begin{equation*}
        \begin{tikzpicture}[baseline= (a).base]
        \node[scale=1] (a) at (0,0){
        \begin{tikzcd}
       (G\times_k F_0)/H \arrow[r] & Z_0\times_{Q_0} E_0
       \end{tikzcd}};
        \end{tikzpicture}
\end{equation*}
is a bijection on $\overline{k}$-points:
\begin{itemize}
    \item The map is surjective. We take $\left ([(g, p)], e\right)\in (Z_0\times_{Q_0} E_0)(\overline{k})$. Let $r\in F_0(\overline{k})$ be an element that is mapped to $e$ and $p$ by $F_0\rightarrow E_0$ and $F_0\rightarrow P_0$, respectively. By construction, the element $[(g,r)]$ is mapped to $([(g, p)], e)$.
    \item The map is injective. Let $[(g,r)]$ and $[(g',r')]$ be in $\left ((G\times_k F_0)/H\right)(\overline{k})$ with the same image in $(Z_0\times_{Q_0} E_0)(\overline{k})$. If we compose with the projection $Z_0\times_{Q_0} E_0\rightarrow E_0$, then we get that $r$ and $r'$ differ by the action of an element of $H(\overline{k})$, and thus we can suppose that $r=r'$. If we compose with the projection $Z_0\times_{Q_0} E_0\rightarrow Z_0$, then we get that $(g,r)$ and $(g',r)$ differ by the action of an element of $\text{Ker}(H(\overline{k})\rightarrow H^{\text{tor}}(\overline{k}))$, and thus $[(g,r)]=[(g',r)]$. 
\end{itemize}
The map 
\begin{equation*}
        \begin{tikzpicture}[baseline= (a).base]
        \node[scale=1] (a) at (0,0){
        \begin{tikzcd}
       (G\times_k F_0)/H \arrow[r] & Z_0\times_{Q_0} E_0
       \end{tikzcd}};
        \end{tikzpicture}
\end{equation*}
is an isomorphism, too. Indeed, by Galois descend, it is enough to prove it over $\overline{k}$. By \cite[Proposition 7.16]{MR1182558} the map is birational, and thus, by Zariski's `Main Theorem' (see \cite[Corollary 4.6]{MR1917232}), the map is an open immersion. Since the map is bijective over $\overline{k}$-points, the map is an isomorphism. We apply the exact sequence of Sansuc to the $H$-torsor $G\times_k F_0 \longrightarrow(G\times_k F_0)/H$ to prove that $(G\times_k F_0)/H$ is $\mathbb G_m$-quasi-trivial.
\begin{equation*}
        \begin{tikzpicture}[baseline= (a).base]
        \node[scale=1] (a) at (0,0){
        \begin{tikzcd}
       0 \arrow[r] & \U((G\times_k F_0)/H)\arrow[r] & \U(G\times_k F_0)\arrow[r] & \widehat H,
       \end{tikzcd}};
        \end{tikzpicture}
\end{equation*}
and
\begin{equation*}
        \begin{tikzpicture}[baseline= (a).base]
        \node[scale=1] (a) at (0,0){
        \begin{tikzcd}
       \widehat H \arrow[r] & \pico {(G\times_k F_0)/H}\arrow[r] & \pico {G\times_k F_0}.
       \end{tikzcd}};
        \end{tikzpicture}
\end{equation*}
By Lemma \ref{san}, the group $\pico {G\times_k F_0}\simeq \pico G \oplus \pico {F_0}$ is trivial, since $G$ is $\mathbb G_m$-quasi-trivial and $F_0$ is a torus. If we prove that $\U(G\times_k F_0)\rightarrow \widehat H$ is surjective with a permutation module as kernel, then we get the statement. By Lemma \ref{rosi}, there is the identity $\U(G\times_k F_0)\simeq \widehat G\oplus \widehat {F_0}$, that implies the surjectivity of the map. The kernel $\U((G\times_k F_0)/H)$ fits into the following commutative diagram with exact rows:
\begin{equation*}
        \begin{tikzpicture}[baseline= (a).base]
        \node[scale=1] (a) at (0,0){
        \begin{tikzcd}
       0 \arrow[r]& \widehat {E_0} \arrow[r]\arrow[d, "\id"]& \U\left((G\times_k F_0)/H\right)\arrow[r]\arrow[d]& \widehat G \arrow[d]\arrow[r]& 0 \\
       0 \arrow[r] & \widehat{E_0} \arrow[r] & \widehat{F_0} \arrow[r]& \widehat H\arrow[r] & 0.
       \end{tikzcd}};
        \end{tikzpicture}
\end{equation*}
The first exact sequence is split, by \cite[Lemma 1]{requiv}. Since $\widehat {E_0}$ and $\widehat G$ are permutation lattices, we can conclude that $\U((G\times_k F_0)/H)$ is, too.
\end{proof}

\begin{oss}
In the previous proposition we cannot apply directly Proposition \ref{colliotprati}, because the $G$-homogeneous space $Z_0\longrightarrow Q_0$ is not principal.
\end{oss}

\begin{prop}
\label{ctk2}
Suppose that $k$ is good and that $G$ is $\mathbb G_m$-quasi-trivial. Let $Y\overset{S}{\longrightarrow} Z_0$ be a flasque quasi-resolution of $(Z_0,z_0)$. The evaluation map
\begin{equation*}
        \begin{tikzpicture}[baseline= (a).base]
        \node[scale=1] (a) at (0,0){
        \begin{tikzcd}
       Z_0(k) \arrow[r, "\eve_{Z_0}"] &  \emph{H}^1(k,S)
       \end{tikzcd}};
        \end{tikzpicture}
\end{equation*}
is surjective.
\end{prop}

\begin{proof}
By Proposition \ref{nonsapreiproprio}, it is enough to prove the statement for a fixed flasque quasi-resolution of $(Z_0,z_0)$. Let $Y_0 \longrightarrow Z_0$ be the pull-back of the $S_0$-torsor $E_0\longrightarrow Q_0$ by $Z_0\rightarrow Q_0$. By Proposition \ref{ctk3}, it is a flasque quasi-resolution of $(Z_0,z_0)$. We take the commutative diagram
\begin{equation*}
        \begin{tikzpicture}[baseline= (a).base]
        \node[scale=1] (a) at (0,0){
        \begin{tikzcd}
       Z_0(k)\arrow[r, "\ev_{Z_0}"]\arrow[d] & \text{H}^1(k,S_0)\arrow[d, "\id"]\\
       Q_0(k)\arrow[r, "\ev_{Q_0}"] & \text{H}^1(k,S_0).
       \end{tikzcd}};
        \end{tikzpicture}
\end{equation*}
The map $Q_0(k)\rightarrow \text{H}^1(k,S_0)$ is surjective, since $\text{H}^1(k,E_0)$ is trivial. If we show that $Z_0(k)\rightarrow Q_0(k)$ is surjective, then we get the statement. Let $q \in Q_0(k)$. We take the fiber product 
\begin{equation*}
        \begin{tikzpicture}[baseline= (a).base]
        \node[scale=1] (a) at (0,0){
        \begin{tikzcd}
       (Z_0)_q\arrow[r]\arrow[d] & Z_0\arrow[d] \\
       \text{Spec}\ k \arrow[r, "q"]& Q_0.
       \end{tikzcd}};
        \end{tikzpicture}
\end{equation*}
We have to prove that $(Z_0)_q(k)$ is non-empty. By computation, the stabilizer of $z_0$ is the kernel $H'$ of $H\rightarrow H^{\text{tor}}$. In particular, it is connected and its maximal toric quotient is trivial. Thus $Z_0\rightarrow Q_0$ is a $G$-homogeneous space, whose geometric stabilizer is connected and have a trivial toric quotient. The same hold for $(Z_0)_q\rightarrow \text{Spec}\ k$ and so, by Proposition \ref{excoro}, $(Z_0)_q(k)$ is non-empty.
\end{proof}

\begin{oss}
In the previous proposition, the surjectivity of $\ev_{Z_0}$ is equivalent to the surjectivity of $Z_0(k)\rightarrow Q_0(k)/R$, since $Q_0(k)/R\rightarrow \text{H}^1(k, S_0)$ is bijective. 
\end{oss}

\begin{proof}[First proof of Proposition \ref{ctkthm}]
By \cite[Lemma 1.5]{MR2237268}, we can suppose $G$ is $\mathbb G_m$-quasi-trivial. We get the claim by Proposition \ref{nuova} and Proposition \ref{ctk2}
\end{proof}

Now we propose the second proof:

\begin{proof}[Second proof of Proposition \ref{ctkthm}]
Without loss of generality, we can assume that $G$ is $\mathbb G_m$-quasi-trivial, and that the flasque quasi-resolution is the one of Subsection \ref{secondco}. Then, it is enough to prove that the map $\Hh^1(k,H)\rightarrow \Hh^1(k,H^{\text{tor}})$ is surjective. Let $\mathcal M\in \text{H}^1(k,H^{\text{tor}})$. By Proposition \ref{excoro} and the fact that $H'\coloneqq \text{Ker} (H\rightarrow H^{\text{tor}})$ has trivial maximal toric quotient, there exists a dominant $H$-equivariant map $\mathcal H\rightarrow \mathcal M$, where $\mathcal H\in \text{H}^1(k,H)$. We note that $\mathcal H \longrightarrow \mathcal M$ is a $H'$-torsor, indeed it is enough to check it on the algebraic closure. Moreover, the map $\mathcal H\longrightarrow (\mathcal H\times_k H^{\text{tor}})/ H$ is an $H'$-torsor, too. By the universal property they are isomorphic and there is a commutative diagram. In particular, $\mathcal M$ is the pushforward of $\mathcal H$ by $H\rightarrow H^{\text{tor}}$, and thus the map is surjective.  
\end{proof}

\appendix

\section{On algebraic groups and torsors}

The results and definitions contained in this section are classical. We reference them because they will be used multiple times in the paper. Of particular interest for us is an exact sequence attributed to Sansuc. For our purpose, we need a more explicit construction than the original one. The second part of the section is dedicated to doing so.

The following lemma and proposition are due to Rosenlicht:

\begin{lemma}
\label{rosi}
Let $X$ and $Y$ be two algebraic $k$-varieties. 
\begin{itemize}
    \item $\Ue(X)$ is a free finitely generated group.
    \item the map $\Ue(X)\oplus \Ue(Y) \rightarrow \Ue(X\times_k Y)$ is an isomorphism.
\end{itemize}
\end{lemma}

\begin{proof}
For the first item see \cite[Theorem 1]{MR133328}. For the second item see \cite[Theorem 2]{MR133328} or \cite[Theorem 4.1]{noteconrad}.
\end{proof}

\begin{prop}
\label{rosigene}
Let $G$ be a connected linear algebraic $k$-group, and let $T$ be an algebraic torus defined over $k$. Let $f:G\rightarrow T$ be a morphism of $k$-schemes. There is a commutative diagram:
\begin{equation*}
        \begin{tikzpicture}[baseline= (a).base]
        \node[scale=1] (a) at (0,0){
        \begin{tikzcd}
        G \arrow[r, "\chi"] \arrow[rd, "f"']& T\arrow[d]\\
          & T,
       \end{tikzcd}};
        \end{tikzpicture}
\end{equation*}
such that the vertical map is the multiplication by an element of $T(k)$ and $\chi$ is a morphism of group $k$-schemes.
\end{prop}

\begin{proof}
We can suppose that $f$ fixes the identity. The claim follows by \cite[Corollary 1.2]{noteconrad}.
\end{proof}

Under the appropriate hypotheses, the Picard group is additive:

\begin{lemma}
\label{san}
Let $X$ and $Y$ be smooth algebraic $k$-varieties. We assume that $Y(k)$ is non-empty, and that $Y$ is $\overline{k}$-rational. Then, the canonical map $\emph{Pic}(X) \oplus \emph{Pic}( Y)\rightarrow \emph{Pic} (X\times_k Y)$ is an isomorphism.  
\end{lemma}

\begin{proof}
See \cite[Lemma 6.6]{MR631309}.
\end{proof}

The previous lemma will be applied when $Y$ is a linear algebraic $k$-group, that is why we need the following lemma:

\begin{lemma}
\label{richiami}
Let $G$ be a connected linear algebraic $k$-group. The underlying variety of $G$ is $k$-unirational and $\overline{k}$-rational.
\end{lemma}

\begin{proof}
See \cite[Theorem 18..2]{Borel} and \cite[Corollary 14.14]{Borel}.
\end{proof}

There is a natural way to associate a $\Gamma_k$-morphism to a torsor. Such a morphism will appear in the exact sequence of Sansuc and it will be used for the definition of universal torsor. Let $X$ be an algebraic $k$-variety, and let $G$ be a linear algebraic $k$-group. We are going to define a map:
\begin{equation*}
        \begin{tikzpicture}[baseline= (a).base]
        \node[scale=1] (a) at (0,0){
        \begin{tikzcd}
       \Hh^1(X,G) \arrow[r, "t"] & \text{Hom}_{\Gamma_k}\left(\widehat G, \pico X\right).
       \end{tikzcd}};
        \end{tikzpicture}
\end{equation*}
Let $Y\longrightarrow X$ be a torsor under a linear algebraic $k$-group $G$, and let $\chi\in \widehat G$. We consider over $\overline k$ the pushforward of the torsor by the character $\chi$:
\begin{equation*}
        \begin{tikzpicture}[baseline= (a).base]
        \node[scale=1] (a) at (0,0){
        \begin{tikzcd}
       \overline{Y}\arrow[r]\arrow[d, "G"] & (\overline{Y}\times_{\overline{k}}\mathbb G_{m,\overline{k}})/\overline{G}\arrow[ld, "\mathbb G_m"]\\
       \overline{X},
       \end{tikzcd}};
        \end{tikzpicture}
    \end{equation*}
where the action of $\overline{G}$ on $\overline{Y}\times_{\overline{k}}\mathbb G_{m,\overline{k}}$ is:
    \begin{equation*}
        \begin{tikzpicture}[baseline= (a).base]
        \node[scale=1] (a) at (0,0){
        \begin{tikzcd}
       ((y,r),g)\arrow[r]& (y\cdot g, \chi(g^{-1})r).
       \end{tikzcd}};
        \end{tikzpicture}
\end{equation*}
Then, the image of the of the class of the $G$-torsor $Y\longrightarrow X$ by $t$ is the class of the pushforward of the $\overline{G}$-torsor $\overline{Y}\longrightarrow \overline{X}$ by $\chi$. We have just defined a morphism in $\text{Hom}_{\Gamma_k}\left(\widehat G, \pico X\right)$. Indeed, the map is $\Gamma_k$-equivariant because the objects are defined over $k$. To prove that the map is a group morphism, it is sufficient to express the torsors in terms of cocycles and to perform the computation.

\begin{Def}
The image by $t$ of a $G$-torsor is called the $type$ of the torsor.
\end{Def}

\begin{oss}
The notion of $type$ of a $G$-torsor has been introduced by Colliot-Th\'{e}l\`{e}ne and Sansuc for groups of multiplicative type, in \cite[Section 2.0]{sketch}. These two definitions coincide, up to a sign, because of \cite[Proposition 1.5.2]{sketch}.
\end{oss}

Let $X$ and $Y$ be smooth algebraic $k$-varieties. Let $Y\longrightarrow X$ be a torsor under a connected linear algebraic $k$-group $G$. In \cite{MR631309}, Sansuc proved the existence of an exact sequence that involves both the functors $\U(*)$ and $\text{Pic}(*)$ applied to $Y\longrightarrow X$: 

\begin{prop}
\label{sansuc}
There is the following exact sequence of $\Gamma_k$-modules:
\begin{equation*}
        \begin{tikzpicture}[baseline= (a).base]
        \node[scale=1] (a) at (0,0){
        \begin{tikzcd}
        1 \arrow[r]& \Ue (X) \arrow[r] & \Ue(Y) \arrow[r] & \widehat G \arrow[r] & \emph{Pic} (\overline X) \arrow[r] & \emph{Pic} (\overline Y) \arrow[r] & \emph{Pic} (\overline G) .
       \end{tikzcd}};
        \end{tikzpicture}
\end{equation*}
\end{prop}

\begin{proof}
See \cite[Proposition 6.10]{MR631309}. In the case in which $G$ is an algebraic torus, there is also \cite[Proposition 2.1.1]{sketch}. 
\end{proof}

Before explicating the maps, let us state the following proposition.

\begin{prop}
\label{mum}
Suppose that $k$ is algebraically closed. Let $f:Z\rightarrow Z'$ be a surjective morphism of smooth algebraic $k$-varieties, and let $\sigma: Z\times_k G\rightarrow Z$ be a right action. If $f$ has the following properties:
\begin{itemize}
    \item $f\circ \sigma = f \circ p$, where $p:Z\times_k G\rightarrow Z$ is the projection map, 
    \item the fibres of $f$ contain at most one orbit under $G$, 
\end{itemize}
then $f$ is universally open and $f:Z\rightarrow Z'$ is the geometric quotient of $Z$ under $G$.
\end{prop}

\begin{proof}
We want to apply \cite[Proposition 0.2]{GIT}. We need to show that $G$ is universally open over $k$ and that for all the algebraically closed fields $l$ the geometric fibres of $f$ over $l$ contain at most one orbit under $G_l$. It is a general fact that $k$-schemes are universally open (see \cite[Lemma 9.5.6]{vakil}). It is enough to show that the fibres are isomorphic to $G$, since we get the claim by base change. We fix $z' \in Z'(k)$ and an element $z\in Z(k)$, such that $f(z)=z'$. By the properties of $f$ the only elements in the fibres of $z'$ are $z\cdot g$, where $g\in G(k)$. Thus we get a $G$-equivariant isomorphism between $G$ and $X_y$.
\end{proof}

From now on, we assume that $Y(k)$ is non-empty. 
\begin{itemize}
    \item $\U(X)\rightarrow \U(Y)$ is the pull-back by $\overline{Y} \longrightarrow \overline{X}$.
    \item We fix $x\in X(k)$, such that the fibre $Y_x$ has a $k$-rational point. In particular there is an isomorphism $Y_x\simeq G$, that defines a closed embedding $i:G\rightarrow Y$. The morphism $\U(Y)\rightarrow \U(G)$ is the pull-back by $i$. By Lemma \ref{rosi}, there is a canonical isomorphism $\U(Y\times_k G)\simeq \U(Y)\oplus \U(G)$. We apply \cite[Lemma 6.4]{MR631309} to get the well definition of the map. 
    \item The map $\widehat G\rightarrow \pico X$ is the $type$ of the torsor.
    \item $\pico X\rightarrow \pico Y$ is the pull-back by $\overline{Y} \longrightarrow \overline{X}$.
    \item We fix $x\in X(k)$, such that the fibre $Y_x$ has a $k$-rational point. In particular there is an isomorphism $Y_x\simeq G$, that defines a closed embedding $i:G\rightarrow Y$. The morphism $\pico Y\rightarrow \pico G$ is the pull-back by $i$. By Lemma \ref{san}, there is a canonical isomorphism $\pic(\overline{Y}\times_{\overline{k}} \overline{G})\simeq \pico{Y}\oplus \pico G$. We apply \cite[Lemma 6.4]{MR631309} to get the well definition of the map. 
\end{itemize}

\begin{oss}
\label{funtoriale}
All the above maps are morphisms of groups.  
\end{oss}

\begin{prop}
\label{gammaequi}
The maps of the exact sequence of Sansuc are $\Gamma_k$-equivariant.
\end{prop}

\begin{proof}
We give a proof only for the connecting map, because the other checks follow by the fact that the torsor is defined over $k$. We have to show that the following diagram commutes:
\begin{equation*}
        \begin{tikzpicture}[baseline= (a).base]
        \node[scale=1] (a) at (0,0){
        \begin{tikzcd}
       \widehat G \arrow[r, "\delta"]\arrow[d,"\sigma \cdot"] & \text{Pic}(\overline{X})\arrow[d, "\sigma \cdot"]\\
       \widehat G \arrow[r, "\delta"] & \text{Pic}(\overline{X}).
       \end{tikzcd}};
        \end{tikzpicture}
\end{equation*}
\begin{itemize}
    \item $\delta \circ \sigma$) The character $\chi$ is sent to the $\mathbb G_m$-torsor $(\overline Y \times_{\overline{k}}\mathbb G_{m,\overline{k}})/\overline{G}\rightarrow \overline{X}$, where the action of $G(\overline{k})$ on $({Y}\times_k\mathbb G_{m})(\overline{k})$ is
\begin{equation*}
        \begin{tikzpicture}[baseline= (a).base]
        \node[scale=1] (a) at (0,0){
        \begin{tikzcd}
       ((y,r), g)\arrow[r] & (y\cdot g, (\sigma \cdot \chi)(g^{-1}) r).
       \end{tikzcd}};
        \end{tikzpicture}
\end{equation*}
    \item $\sigma \circ \delta$) The $\mathbb G_m$-torsor $(\overline{Y}\times_{\overline{k}} \mathbb G_{m,\overline{k}})/\overline{G}\rightarrow \overline{X}$ is given by 
\begin{equation*}
        \begin{tikzpicture}[baseline= (a).base]
        \node[scale=1] (a) at (0,0){
        \begin{tikzcd}
       \overline{Y}\arrow[r]\arrow[d, "G"] & (\overline{Y}\times_{\overline{k}}\mathbb G_{m,\overline{k}})/\overline{G}\arrow[ld, "\mathbb G_m"] & (\overline{Y}\times_{\overline{k}}\mathbb G_{m,\overline{k}})/\overline{G}\arrow[ld, "\mathbb G_m"]\arrow[l, "(\sigma^{-1})^*"] \\
       \overline{X} & \overline{X} \arrow[l, "(\sigma^{-1})^*"]
       \end{tikzcd}};
        \end{tikzpicture}
\end{equation*}
where the action of $G (\overline{k})$ on $({Y}\times_k\mathbb G_{m})(\overline{k})$ is
\begin{equation*}
        \begin{tikzpicture}[baseline= (a).base]
        \node[scale=1] (a) at (0,0){
        \begin{tikzcd}
       ((\sigma^{-1}(y),\ \sigma^{-1}(r)), g)\arrow[r] & (\sigma^{-1}(y)\cdot \sigma^{-1}(g),\ (\sigma^{-1} \chi(g^{-1}))\ \sigma^{-1}(r)).
       \end{tikzcd}};
        \end{tikzpicture}
\end{equation*}
\end{itemize}
The two actions of $G(\overline{k})$ on $({Y}\times_k\mathbb G_{m})(\overline{k})$ differ by $\sigma^{-1}$, so the quotient spaces induce the same element in $\pico X$.
\end{proof}

\begin{lemma}
\label{pulizia}
Let $f:Y\rightarrow \mathbb G_m$ be a morphism of algebraic $k$-varieties, and let $\chi:G_{\overline{k}}\rightarrow \mathbb G_{m,\overline{k}}$ be the image of $f$ under $\Ue(Y)\rightarrow \widehat G$. There is the following identity (on $\overline{k}$-points):
$$f(y\cdot g)=f(y)\chi(g).$$
\end{lemma}

\begin{proof}
Without loss of generality, we can assume that $k=\overline{k}$. If we fix a point $y\in Y(k)$, then we get an inclusion $i_y:G\subset Y$. By \cite[Lemma 6.4]{MR631309}, the group morphism $\chi$ coincides with $\overline{r_y}\circ f \circ i_y$, where $\overline{r_y}$ is the multiplication in $\mathbb G_m$ by a constant $r_y$. The following computation proves the claim:
$$f(y\cdot g)=(f\circ i_y)(g)=r_y^{-1}\chi(g)=f(y)\chi(g),$$
where the last identity follows by $\text{id}=\chi(e)=(\overline{r_y}\circ f \circ i_y)(e)=r_yf(y).$
\end{proof}

Let us start demonstrating that the sequence is exact.

\begin{lemma}
\label{lemma1}
The sequence Of Sansuc is exact in $\widehat G$.
\end{lemma}

\begin{proof}
Without loss of generality, we can suppose that $k=\overline{k}$.
\begin{itemize}
    \item The sequence $\U(Y)\rightarrow \widehat G \rightarrow \text{Pic}( X)$ is a complex. Let $f:Y\rightarrow \mathbb G_m$ be a morphism of algebraic $k$-varieties, and let $\chi:G\rightarrow \mathbb G_m$ the induced character. We recall that $(Y\times_k \mathbb G_m)/G$ is the quotient of $(Y\times_k \mathbb G_m)$ by the following action of $G$:
    \begin{equation*}
        \begin{tikzpicture}[baseline= (a).base]
        \node[scale=1] (a) at (0,0){
        \begin{tikzcd}
        ((y,r),g)\arrow[r]& (y\cdot g,\ \chi(g^{-1})r).
       \end{tikzcd}};
        \end{tikzpicture}
\end{equation*}
    We define a morphism $h:Y\times_k\mathbb G_m\rightarrow X\times_k \mathbb G_m$ in the following way:
\begin{equation*}
        \begin{tikzpicture}[baseline= (a).base]
        \node[scale=1] (a) at (0,0){
        \begin{tikzcd}
        (y,r) \arrow[r] & (p(y), f(y)r),
       \end{tikzcd}};
        \end{tikzpicture}
\end{equation*}
    where $p$ is the $G$-torsor $Y\longrightarrow X$. Such a map is a $\mathbb G_m$-equivariant map of schemes over $X$. Furthermore, it is $G$-invariant:
    $$h(y,r)=(p(y),f(y)r)\overset{?}{=}(p(y\cdot g), f(y\cdot g)\chi(g^{-1})r)=h(y\cdot g,\ \chi(g^{-1})r),$$
    where the central identity follows by Lemma \ref{pulizia} and $G$-invariance of the torsor $p$. By Proposition \ref{mum}, there is an isomorphism of $\mathbb G_m$-torsors $(Y\times_k \mathbb G_m)/G\rightarrow X\times_k \mathbb G_m$ over $X$, which is the claim.
\item The sequence $\U(Y)\rightarrow \widehat G \rightarrow \text{Pic} (X)$ is exact. Let $\chi: G\rightarrow \mathbb G_m$ be in the kernel of $\delta$. This means that the pushforward of $Y\longrightarrow X$ by $\chi$ is the trivial torsor. In particular, there is an isomorphism $(Y\times_k\mathbb G_m)/G\overset{\sim}{\longrightarrow}X\times_k \mathbb G_m$, that we use to define a map $Y\rightarrow\mathbb G_m$:
    \begin{equation*}
        \begin{tikzpicture}[baseline= (a).base]
        \node[scale=1] (a) at (0,0){
        \begin{tikzcd}
        Y\arrow[r] & (Y\times_k\mathbb G_m)/G \arrow[r, "\sim"] & X\times_k \mathbb G_m \arrow[r] & \mathbb G_m,
       \end{tikzcd}};
        \end{tikzpicture}
    \end{equation*}
where the morphism $Y\rightarrow (Y\times_k \mathbb G_m)/G$ is the inclusion to the first component and the map $X\times_k \mathbb G_m\rightarrow \mathbb G_m$ is the projection to the second component. If we fix a point $y\in Y(k)$, then we get an inclusion $i_y:G\subset Y$. By construction, up to a sign and up to multiply in $\mathbb G_m$ by a constant, the composition map $G\rightarrow Y \rightarrow \mathbb G_m$ coincides with the character $\chi$.
\end{itemize}
\end{proof}

\begin{lemma}
\label{lemma2}
The sequence of Sansuc is exact in $\emph{Pic}(\overline{X})$.
\end{lemma}

\begin{proof}
Without loss of generality, we can assume that $k=\overline{k}$.
\begin{itemize}
    \item The sequence $\widehat G \rightarrow \text{Pic}( X) \rightarrow \text{Pic}( Y)$ is a complex. Let $\chi:G\rightarrow \mathbb G_m$ be a character. By functoriality, there is the following commutative diagram:
\begin{equation*}
        \begin{tikzpicture}[baseline= (a).base]
        \node[scale=1] (a) at (0,0){
        \begin{tikzcd}
        \text{H}^1(X,G)\arrow[r, "\chi_{*}"]\arrow[d]&\text{H}^1(X,\mathbb G_m)\arrow[d]\\
        \text{H}^1(Y,G)\arrow[r, "\chi_{*}"] & \text{H}^1(Y,\mathbb G_m).
       \end{tikzcd}};
        \end{tikzpicture}
\end{equation*}
The map $\Hh^1(X,G)\rightarrow \Hh^1(Y,G)$ trivialises the $G$-torsor $Y\longrightarrow X$. Thus, following the other direction we obtain that the image of $\widehat G \rightarrow \text{Pic}( X)$ is contained in the kernel of $\text{Pic}( X)\rightarrow \text{Pic}( Y)$.
\item The sequence $\widehat G \rightarrow \text{Pic}( X) \rightarrow \text{Pic}( Y)$ is exact. Let $L\longrightarrow X$ be a $\mathbb G_m$-torsor that is trivialised by $Y\longrightarrow X$. We have to show that $L\longrightarrow X$ is induced by a character. There is the following cartesian diagram:
\begin{equation*}
        \begin{tikzpicture}[baseline= (a).base]
        \node[scale=1] (a) at (0,0){
        \begin{tikzcd}
        Y\times_k \mathbb G_m\simeq L_Y \arrow[r, "G"]\arrow[d, "\mathbb G_m"] & L \arrow[d, "\mathbb G_m"]\\
        Y \arrow[r, "G"]& X.
       \end{tikzcd}};
        \end{tikzpicture}
\end{equation*}
We fix $l\in L(k)$. The fibre over $l$ of $Y\times_k\mathbb G_m\longrightarrow L$ is isomorphic to $G$, thus we can define a morphism:
\begin{equation*}
        \begin{tikzpicture}[baseline= (a).base]
        \node[scale=1] (a) at (0,0){
        \begin{tikzcd}
        G \arrow[r] & Y\times_k \mathbb G_m \arrow[r] & \mathbb G_m,
       \end{tikzcd}};
        \end{tikzpicture}
\end{equation*}
by composition with the projection $Y\times_k \mathbb G_m\rightarrow \mathbb G_m$. If we change the trivialisation of $L_Y\longrightarrow Y$ or if we take a second $G$-isomorphism between the fibre over $l$ and $G$, then we get two maps $G\rightarrow \mathbb G_m$ that differ up to multiply in $\mathbb G_m$ by a constant. We have to show that the choice of the point $l$ does not change the map, up to multiply in $\mathbb G_m$ by a constant. We do it in two steps:
\begin{itemize}
    \item The map does not depend from the chosen open subset $U$ of $X$. Let $U$ be an open subset that contains the image of $l$ by $L\longrightarrow X$. There is the following commutative diagram, where all the squares are cartesian:
    \begin{equation*}
        \begin{tikzpicture}[baseline= (a).base]
        \node[scale=1] (a) at (0,0){
        \begin{tikzcd}
         & & G \arrow[rr]\arrow[dd]\arrow[rd] & & \text{Spec}\  k\arrow[dd]\arrow[rd, "l"]\\
        \mathbb G_m  & Y\times_k\mathbb G_m\arrow[rr, "\sim\ \ \ \ \ \ \ \ \ \ "]\arrow[l] & & Y\times_X L\arrow[rr]\arrow[dd] & & L\arrow[dd] \\
         Y_U\times_k \mathbb G_m\arrow[u]\arrow[rr, "\sim"]\arrow[ru] & & Y_U \times_U L_U \arrow[rr]\arrow[dd]\arrow[ru]& & L_U \arrow[dd]\arrow[ru]\\
        & & & Y\arrow[rr] & & X\\
         & & Y_U \arrow[ru]\arrow[rr]& & U\arrow[ru]\\
       \end{tikzcd}};
        \end{tikzpicture}
    \end{equation*}
    The diagram shows that if we replace $X$ by an open subset, then the map $G\rightarrow \mathbb G_m$ does not change. In particular, since $X$ is connected, it is enough to show that the map does not change over a chosen open covering of $X$.
    \item Let $U$ be an open subset of $X$, that trivialises the $\mathbb G_m$-torsor $L\longrightarrow X$ and that contains the image of $l$ under the same map. We can suppose that $U=X$, up to replace $X$ with $U$. There is the following commutative diagram, where all the squares are cartesian:
\begin{equation*}
        \begin{tikzpicture}[baseline= (a).base]
        \node[scale=1] (a) at (0,0){
        \begin{tikzcd}
        && G\arrow[rr]\arrow[d] && \text{Spec}\ k\arrow[d] \\
        && Y\arrow[rr]\arrow[dd] && X\arrow[dd] \\
        \mathbb G_m & Y \times_k \mathbb G_m \arrow[rr]\arrow[dr, "\sim"]\arrow[l]&& X\times_k \mathbb G_m\arrow[dr, "\sim"] \\
        && L_Y\arrow[rr]\arrow[d] && L\arrow[d] \\
        && Y \arrow[rr]&& X.
       \end{tikzcd}};
        \end{tikzpicture}
\end{equation*}
We can suppose that $\text{Spec}\ k\rightarrow L$ coincides with $l$, up to change the trivialisation of $L\longrightarrow X$. The diagram shows that the induced map $G\rightarrow \mathbb G_m$ factors trough $Y$. By \cite[Lemma 6.4]{MR631309}, we get the independence by the chosen point.
\end{itemize}
By Proposition \ref{rosigene}, up to multiply in $\mathbb G_m$ by a constant, the above map $\chi:G\rightarrow \mathbb G_m$ is a group morphism. We have to show that $L\longrightarrow X$ is the image of $\chi$ under $\widehat G \rightarrow \text{Pic}(X)$. By construction, the action of $G$ on $Y\times_k\mathbb G_m$ is:
\begin{equation*}
        \begin{tikzpicture}[baseline= (a).base]
        \node[scale=1] (a) at (0,0){
        \begin{tikzcd}
       ((y,r),g)\arrow[r]& (y\cdot g, \chi(g^{-1})r).
       \end{tikzcd}};
        \end{tikzpicture}
\end{equation*}
By Proposition \ref{mum}, there is an isomorphism of $\mathbb G_m$-torsors over $X$ between $L$ and $(Y\times_k\mathbb G_m)/G$.
\end{itemize}
\end{proof}

A combination of previous lemmas and known results produce the following:

\begin{prop}
The sequence of Sansuc is exact.
\end{prop}

\begin{proof}
By \cite[Lemma 6.4]{MR631309}, almost all the maps that we have defined coincide with the ones of \cite[Proposition 6.10]{MR631309}, that proves that the sequence is exact in $\U(X)$, $\U(Y)$ and $\pico Y$. By Lemma \ref{lemma1} and Lemma \ref{lemma2}, we get the statement.
\end{proof}

\printbibliography[
heading=bibintoc,
title={References}
]

\Addresses

\end{document}